%% file: mainformatted.tex
\documentclass[11pt, a4paper]{article}

\usepackage[round,comma]{natbib}
\usepackage[ruled,lined]{algorithm2e}
\usepackage[utf8]{inputenc}
\usepackage{hyperref}
\usepackage{mathtools}
\input{declarations.tex}
\usepackage[ruled,lined]{algorithm2e}
\newcommand{\capped}[1]{\ensuremath{\bar{#1}^{\textsc{cap}}}}

\newcommand{\mean}{\ensuremath{\theta}}
\newcommand{\deltamm}{\ensuremath{\delta^{\textsc{mm}}}}
\newcommand{\deltabayes}{\ensuremath{\delta^{\textsc{Bayes}}}}

\newcommand{\meanplus}{\ensuremath{\mean^{+}}}
\newcommand{\meanmin}{\ensuremath{\mean^{-}}}

\newcommand{\curry}{\ensuremath{y}}
\newcommand{\currY}{\ensuremath{Y}}
\newcommand{\currz}{\ensuremath{}}

\newcommand{\aspace}{\ensuremath{\cA}}
\newcommand{\action}[1]{\ensuremath{a}}
\newcommand{\dr}[1]{\ensuremath{\delta}}
\newcommand{\drbayes}[1]{\ensuremath{\deltabayes}}
\newcommand{\drmm}[1]{\ensuremath{\deltamm}}

\newcommand{\bbw}{\ensuremath{\text{\sc bbw}}}

\newcommand{\sd}{\ensuremath{\text{\sc sd}}}
\newcommand{\lr}{\ensuremath{\text{\sc lr}}}

\newcommand{\Sl}{\ensuremath{S^{\textsc{l}}}}
\newcommand{\Sr}{\ensuremath{S^{\textsc{r}}}}
\newcommand{\Slr}{\ensuremath{S^{\textsc{lr}}}}

\newcommand{\longp}[1]{}
\newcommand{\shortp}[1]{}

\renewcommand\Exp{\ensuremath{\mathbf E}}

\newcommand{\Cb}{\ensuremath{\mathcal{B}}}

\begin{document}

\title{The E-Posterior}
\author{Peter D. Gr\"unwald\footnotemark[1]}
\footnotetext[1]{CWI, Amsterdam, and  Mathematical Institute, Leiden University. CWI is the National Resarch Institute for Mathematics and Computer Science in the Netherlands.}
\maketitle
\commentout{

\address{$^{1}$CWI, Amsterdam, and Leiden University. ORCID ID 0000-0001-9832-9936}

\subject{Mathematics $\rightarrow$ Statistics}

\keywords{Bayes-frequentist debate, conditional inference, e-values, decision making, uncertainty quantification, Savage-Dickey ratio \\
\ \\ CWI is the National Research Institute for Mathematics and Computer Science in the Netherlands. A part of this paper overlaps with a small part (Section 4.2 and paragraphs in introductory sections announcing it) of the first, unsubmitted arXiv version of the paper {\em Beyond Neyman-Pearson\/} \citep{Grunwald22}. This part will be removed from the second and all future arXiv versions and will not be included in the  journal submission of \citep{Grunwald22}.
}

\corres{Peter D. Gr\"unwald\\
\email{pdg@cwi.nl}}
}
\begin{abstract}
We develop a representation of a decision maker's uncertainty based on e-variables. Like the Bayesian posterior, this {\em e-posterior\/} allows for making predictions against arbitrary loss functions that may not be specified ex ante.  Unlike the Bayesian posterior, it provides risk bounds that have frequentist validity 
irrespective of  prior adequacy: if the e-collection (which plays a role analogous to the Bayesian prior)  is chosen badly, the bounds get loose rather than wrong, making {\em e-posterior minimax\/} decision rules safer than Bayesian ones. The resulting  {\em quasi-conditional paradigm\/} is illustrated by re-interpreting a previous influential  partial Bayes-frequentist unification, {\em Kiefer-Berger-Brown-Wolpert conditional frequentist tests}, in terms of e-posteriors.
\end{abstract}



\noindent 
{This version is essentially  identical to the one published in the \\ {\em Philosophical Transactions of the Royal Society A}, Volume 381, Issue 2247, 2023. The only difference is the addition of a few references in Section 1, underneath (4), and in Section 5.}
\section{Introduction and Overview}
It is one of the main attractions of Bayesian inference: one may assess the posterior-expected performance of arbitrary decision rules in arbitrary decision tasks, and, as a consequence, perform posterior-optimal decisions for any such task. An important special case arises if we are given a statistical model $\cP = \{P_{\theta} : \theta \in \Theta \}$ equipped with a prior distribution $W$ on $\Theta$,  a space of actions $\aspace$ and a  loss function $L: \Theta \times \aspace \rightarrow \reals$, and we observe data $Y=y$, with $Y\sim P_{\theta}$ for some $\theta \in \Theta$, taking values in some set $\cY$. We call such a combination of $\cP$ and $L$ a {\em decision problem}.  
A decision rule $\dr{b}$ is then any function mapping data $y \in \cY$ to action $a \in \aspace$; we write $\dr{b}(y) = a$. We then assess the expected performance of $\dr{b}$ given data $Y=y$ as
\begin{align}\label{eq:bayesassessment}
{\bf E}_{\theta \sim W \mid y} [L_{\currz}(\theta,\dr{b}(y))],   
\end{align}
with $W \mid y$ denoting the Bayes posterior distribution based on prior $W$ and data $Y=y$. 
In their classic {\em Bayesian theory\/} (\citeyear{BernardoS94}), Jos\'e Bernardo and Adrian Smith adopt, as they make clear already on page 2, a  {\em wholehearted subjectivist position\/} on the interpretation of probability. Such a subjective Bayesian stance implies that if {\em the prior distribution truly describes one's initial degrees of belief}, then  the assessment (\ref{eq:bayesassessment}) is not just correct for {\em any\/} decision 
rule $\dr{b}$, it also implies optimality of the {\em Bayes decision rule}, defined to be any rule $\deltabayes$ such that, for all $y \in \cY$, 
\begin{align}\label{eq:bayesdecision}
    \bar{R}^{\textsc{bayes}}(W \mid y)  := {\bf E}_{\theta \sim W \mid y} [L_{\currz}(\theta,\drbayes{b}(y))] = 
    \min_{a \in \aspace} {\bf E}_{\theta \sim W \mid y} [L_{\currz}(\theta,a)].
\end{align}
Unfortunately, in practice, for lack of time, computational power or imagination, one often --- in fact, nearly always --- uses pragmatic priors in combination with potentially `misspecified' (wrong) models. Then the priors do not truly represent one's beliefs, and as a consequence  (\ref{eq:bayesassessment}) may give a highly misleading impression of $\dr{b}$'s quality in the real world, and using $\deltabayes$ may lead to disastrous results. Much of my own research in the past has concerned misspecified models \citep{Grunwald99a,GrunwaldO17}. In the present paper, I assume the model is correct and concentrate on the case in which "good" priors (i.e.\ sufficiently sophisticated to justify using (\ref{eq:bayesdecision}) in practice) are hard to come by. 
The pragmatic priors that are then used often lead to  {\em overconfident posteriors}, examples of which abound in the literature: decisions based on $W \mid y$ are then substantially worse in practice than predicted by $W \mid y$ itself via (\ref{eq:bayesassessment}) \citep{oelrich2020overconfident}, especially in nonparametrics \citep{szabo2015frequentist}; naturally such assessments can lead to (very) bad decisions \citep{balch2019satellite}. In this paper, we provide the {\em e-posterior\/} as an alternative for such cases: while remaining close to the Bayesian ideal, it  has a 
frequentist justification that avoids overconfidence and other misleading assessments.
As an additional advantage, in contrast to Bayesian methods, e-posteriors are readily used in highly nonparametric situations in which a likelihood is not easy to define --- we defer discussion of this bonus until the discussion Section~\ref{sec:discussion}.  
\paragraph{The E-Posterior}
The e-posterior is an analogue to the Bayes posterior for which frequentist analogues ((\ref{eq:riskassessment}) and (\ref{eq:minimax}) below) of Bayesian assessments ((\ref{eq:bayesassessment}) and (\ref{eq:bayesdecision}) above) hold simultaneously under all $\theta \in \Theta$. It is based on e-values, a far-reaching generalization of likelihood ratios. Crucially, just as with the Bayesian posterior, the same e-posterior can be combined with arbitrary decision problems and loss functions. That is, if a given prior $W$ truly represents beliefs, the Bayesian (\ref{eq:bayesassessment}) and (\ref{eq:bayesdecision}) are appropriate simultaneously for arbitrary sets of loss functions $\{L_{b}: b \in \Cb\}$, substantiating the common Bayesian claim that the Bayesian posterior summarizes all one's uncertainty about the parameter given the data --- since all conceivable decisions can be based on it. The same holds for the e-posterior analogues below: they are simultaneously valid for all loss functions, so that, analogously to the Bayesian case, the e-posterior can also be thought of as {\em a single function summarizing one's uncertainty}. As discussed in Section~\ref{sec:riskassessment}, this is what sets out the e-posterior from  previous attempts to provide frequentist guarantees for Bayesian posteriors or frequentist analogues thereof such as confidence distributions. There we also explain why  we regard validity under post-hoc choice of loss function as crucial in statistical testing and estimation in the applied sciences.

Both the Bayesian and the e-representation of uncertainty are relative: whereas the Bayesian posterior is defined relative to a prior defined upon a statistical model $\cP$,  the e-posterior is defined relative to a 
collection $\cS$ of e-variables, which is itself defined relative to $\cP$.
This `e-collection' $\cS$ in turn may itself be implicitly based on a prior (Example~\ref{ex:normal}) or a set of priors
(Section~\ref{sec:generalizedsd}). Because it gives rise to stochastic upper bounds on risks rather than precise risks, we  are able to obtain the result that the e-posterior is {\em valid\/} in a frequentist sense irrespective of the collection $\cS$ taken; but the bounds may get better or worse depending on 
$\cS$. This makes it considerably more robust than the Bayesian risk assessment (\ref{eq:bayesassessment}), which, at least under a subjective Bayes interpretation, one can only trust if one truly fully believes one's prior $W$ and one's model $\cP$. It puts the e-posterior approach in the {\em luckiness\/}-framework 
\citep{Grunwald07,ShaweTaylorW97,HerbrichW02a,GrunwaldM19} in which Bayesian {\em belief\/} (one has to believe one's prior to trust one's assessment) is replaced by {\em guarantees\/} and {\em hope\/} that the guarantees are 
good: if the prior is well-aligned with the data, one is {\em lucky\/} and gets tight bounds; if not, the bounds one gets are loose but still valid. It goes substantially further though than standard `luckiness' approaches (such as PAC-Bayesian methods \citep{GrunwaldM19}), since these invariably incorporate an a-priori given loss function into the likelihood (e.g. via a Gibbs posterior) so that the same (e.g. Gibbs) posterior can only be used for a single loss function. Our e-posterior being combinable with arbitrary loss functions, we chose to call the resulting type of inference {\em quasi-conditional}, a name we explain further in Section~\ref{sec:riskassessment}. 
Before we start with the details, we should point out that, of course, not all can be rosy: the specific type of stochastic upper bound that provides our sense of quasi-conditional frequentist validity, as given in Proposition~\ref{prop:TheBound} and interpreted underneath in terms of a De Finetti-type of game between a decision-maker and a bookie, while holding universally, may in some instances provide a weaker inference than one would really like to see --- we consider it as the main open problem of this work to further investigate this novel type of validity. 
\paragraph{Formal Definitions}
Formally, we associate $\cP = \{P_{\theta} : \theta \in \Theta\}$ with some corresponding collection $\cS$ of e-variables
$\cS= \{S_{\theta} : \theta \in \Theta\}$. 
We focus on what we call the {\em simple setting}, in which all parameters in parameter vector $\theta \in \Theta \subseteq \reals^k$ are of interest. An e-variable $S_{\theta}$ relative to any $\theta \in \Theta$ is then simply  any  {\em nonnegative\/}
statistic (i.e.\ a random variable that can be written as a function of the data, $S_{\theta}= S_{\theta}(Y)$) that satisfies the inequality:
\begin{equation}\label{eq:basic}
{\bf E}_{Y \sim P_{\theta}}[S_{\theta}] \leq 1.
\end{equation}
We defer the case of nonparametrics and nuisance parameters, requiring a slightly extended definition, to Section~\ref{sec:discussion}.
Over the last three years, interest in  e-variables has exploded  \citep{GrunwaldHK19,wasserman2020universal,Shafer:2021,VovkW21,henzi2021valid,ramdas2022testing}, as tools to extend traditional Neyman-Pearson tests to situations with optional stopping and continuation while keeping Type-I error guarantees --- see  Section~ \ref{sec:discussion} for details. Here we use them in a novel way: 
we define the {\em e-posterior\/} based on $\cS$ and $Y$ to be the reciprocal of $S_{\theta}$:
\begin{align}
    \label{eq:eposterior}
    \bar{P}(\theta \mid y) := \frac{1}{S_{\theta}(y)},
\end{align}
where by convention we set $1/0$ to be $\infty$. 
We note that the idea to study the reciprocal of $S_{\theta}$ directly is in itself not new --- it has  been considered earlier under the name of `confidence distribution' by \cite[Appendix D5]{waudby2020estimating}, the first version of which came out in 2020 (with $\theta \in [0,1]$ the mean of a bounded random variable, and already including figures similar to our Figure~\ref{fig:eposterior}) and also by \cite{PawelLW22} to derive `evidential' rather than `credible' posterior intervals. The novelty in this paper is to take (\ref{eq:eposterior}) center- stage and to analyze its decision-theoretic properties. 
To this end, we continue with an entirely novel definition:  we set the {\em e-posterior risk assessment\/} of decision rule $\delta$ based on $\cS$ to be: 
\begin{align}\label{eq:riskassessment}
 \textbf{e-posterior risk bound:} \ 
 \bar{R}(\delta) := \bar{R}(Y,\dr{B}(Y)) \text{\ with\ } 
     \bar{R}(y,a) := \sup_{\theta \in \Theta} \bar{P}(\theta \mid y)  \cdot L_{\currz} (\theta,a), 
\end{align} 
where we use the convention that $\infty\cdot a := 0$ if $a=0$ and $\infty$ if $a > 0$. 
\commentout{The idea is that a decision maker (DM) is  presented a decision problem with loss function $L_{\currz}$, parameterized by some $b$; this is formalized as DM  observing $B=b$, where the distribution of $B$, determining the loss of interest, will typically be unknown to DM, and may depend on the observed $Y=y$ (for illustration see Section~\ref{sec:gaussexample} and various examples thereafter). DM then performs action $\dr{b}(y)$;} 
$\bar{R}(\delta)$ is the random variable that maps $Y$ 
to the corresponding risk bound
(\ref{eq:riskassessment}), which  is to be interpreted as a specific type of {\em stochastic upper bound\/} on the risk that holds under all $\theta \in \Theta$ simultaneously --- Proposition~\ref{prop:TheBound} below makes this precise. 
It holds for all loss functions $L: \Theta \times \cA \rightarrow \reals$ satisfying the condition that $\sup_{\theta \in \Theta} L(\theta,a) \geq 0$ for each $a$. We call this the {\em no-sure-gain condition}. It is satisfied if $L\geq 0$, but in general it also allows negative losses (as will be useful in Example~\ref{ex:science}). 
In contrast to the upper bound-nature of (\ref{eq:riskassessment}), the Bayesian (\ref{eq:bayesassessment})  does not just bound but precisely gives the conditional risk.  Next, in analogy to (\ref{eq:bayesdecision}), we suggest, once $\cS$ has been fixed, the $\cS$-based {\em e-posterior-minimax decision rule\/} as any rule $\deltamm$ satisfying, for all $y \in \cY$, any loss function $L_{\currz}$,
\begin{align}\label{eq:minimax}
\textbf{e-posterior minimax rule:} \ 
\sup_{\theta \in \Theta} \ \bar{P}(\theta \mid y) \cdot L_{\currz}(\theta,\drmm{b}(y)) =
\min_{a \in \aspace} \sup_{\theta \in \Theta} \ \bar{P}(\theta \mid y) \cdot L_{\currz}(\theta,a).
\end{align}

\paragraph{Contents of this Paper}
Section~\ref{sec:riskassessment} provides  Proposition~\ref{prop:TheBound}  which establishes the frequentist sense in which  $\bar{R}(\delta)$ as defined by (\ref{eq:riskassessment}) is valid. We discuss this {\em quasi-conditional\/} validity at length and provide extended examples. 
The paper continues in Section~\ref{sec:taxonomy} by listing various types of e-collections, and investigating the risk assessment and the corresponding e-posterior minimax rules that ensue from them when instantiated to simple yet important situations. In particular, Section~\ref{sec:generalizedsd} considers the Gaussian location family and a weighted squared error loss, with data-dependent weights. In this situation (extended to 1-dimensional exponential families in the Supplementary Material), Bayesian risk assessment based on  standard objective Bayes posteriors can fail spectacularly if the weights strongly vary with the data, whereas our type of risk assessment via (\ref{eq:riskassessment}) is always valid; moreover, if the weights are constant, 
the assessment for 
e-posterior minimax rules can  match minimax optimal frequentist and Bayesian rates up to a small constant factor, thereby alleviating concerns that the $\sup$ in (\ref{eq:riskassessment}) may make assessments overly pessimistic. In Section~\ref{sec:conditionalistfrequentist}, we consider a Bayesian-frequentist unification for hypothesis testing with a simple null that 
was introduced in an influential paper by Berger, Brown and Wolpert (\citeyear{BergerBW94}) (BBW from now on),
building on earlier work by 
\cite{kiefer1976admissibility,Kiefer77,brownie1977ideas}.
BBW's approach was later extended with various collaborators \citep{BergerBW94,Wolpert96b,BergerG01,DassB03},  culminating in Berger's (\citeyear{Berger03}) IMS Lecture paper {\em Could Fisher, Neyman and Jeffreys have agreed on testing?\/} 
We show how to re-interpret the BBW conditional error probabilities, as well as Kiefer's earlier {\em conditional confidence estimators\/}  (\citeyear{kiefer1976admissibility,Kiefer77}) and Vovk's (\citeyear{Vovk93}) extension of it, in terms of (quite special) e-posteriors.
Our re-interpretation suggests that in practice, we will usually want to use other e-variable collections than the BBW one --- which  is not at all  to criticize their pioneering and highly original work, that, together with Vovk's (\citeyear{Vovk93}), has served as the main inspiration for this paper. 
Then, in Section~\ref{sec:capped}, we describe a variation of the e-posterior risk assessment (\ref{eq:riskassessment}) in which the e-posterior is replaced by its {\em capped\/} version (see Figure~\ref{fig:eposterior}) and we investigate when this arguably more intuitive form of the e-posterior gives rise to the same assessment. This turns out to be the case under {\em Condition Zero}, a fundamental condition on the decision rules employed, which may or may not be enforceable in any given decision problem.
Finally, in Section~\ref{sec:discussion}, we give more background on the traditional use and interpretation of e-variables and generalize to testing with composite nulls, estimating with nuisance parameters and nonparametrics, and we provide some concluding discussion.
But first, we introduce several running examples of  e-posteriors, highlighting via (\ref{eq:maxmean}) the connection with the Bayes posterior.
\subsection{Examples of E-Collections}
\begin{example}{\bf [A\  Simple E-Posterior: Savage-Dickey]}\label{ex:simplesd}
A simple way to design e-variables is to start with a prior distribution $W$ on $\Theta$ and setting $p_W(y) := \int p_{\theta}(y) dW(\theta)$ to be the Bayes marginal based on $W$. Here, as in the rest of the paper, we assume that for all $\theta \in \Theta$, $P_{\theta}$ has density $p_{\theta}$ relative to some common underlying measure $\mu$. One then defines, for each $\theta \in \Theta$,
\begin{align}
    \label{eq:sd}
    S^{\sd}_{\theta} =  \frac{p_W(y)}{p_{\theta}(y)} \ \ ; \ \ \bar{P}^{\sd}(\theta \mid y) = 
    \frac{p_{\theta}(y)}{p_W(y)} = \frac{w(\theta \mid y)}{w(\theta)}
\end{align}
where the latter equality holds if $W$ has density $w$ and we denote prior/posterior densities by $w(\theta)$ and $w(\theta \mid y)$, respectively.  $S^{\sd}_{\theta}$ is an e-variable, since it satisfies ${\bf E}_{Y \sim P_{\theta}}[S^{\sd}_{\theta}] = \int_{\cY} p_W(y) d y = 1$. 
\cite{PawelLW22} study this e-variable, focusing on an `evidential' interpretation; it also figures prominently in \citep{GrunwaldHK19} and, under the name `prior-to-posterior e-variable', in \citep{neiswanger2021uncertainty}. We chose our alternative name 
because, in the Bayesian literature, the quantity $w(\theta \mid y)/{w(\theta)}$ is generally known as the {\em Savage-Dickey density ratio\/} after the prominent Bayesians L. Savage and J. Dickey.
The Savage-Dickey e-variables are by no means the only reasonable ones to base an e-posterior on, but they are educationally insightful because they give a first interpretation of risk assessment (\ref{eq:riskassessment}). Namely, using Bayes' theorem we can trivially but unusually rewrite  the {\em standard\/} Bayesian uncertainty assessment (\ref{eq:bayesassessment}) as follows:
\begin{align}
     \Exp_{\theta \sim W \mid y} [L_{\currz}(\theta,\dr{b}(y) )] =
\int \frac{p_\theta(y)}{p_W(y)} \cdot  L_{\currz}(\theta,\dr{b}(y) ) d W(\theta)
=   &  \ \   \Exp_{\theta \sim W } 
[\bar{P}^{\textsc{sd}} (\theta \mid y) \cdot L_{\currz}(\theta,\dr{b}(y) )],\label{eq:meanmax} \\
 \text{whereas with $Y=y$, 
 we have \ } 
\bar{R}(\delta) = & \ \ \ \max_{\theta} \ \ \bar{P}^{\textsc{sd}} (\theta \mid y) \cdot L_{\currz}(\theta,\dr{b}(y) ).\label{eq:maxmean}
\end{align}
Replacing prior expectation (\ref{eq:meanmax}) by maximum (\ref{eq:maxmean}), the Savage-Dickey e-posterior is always more pessimistic than the Bayesian posterior based on the same prior, and might thus not look so outlandish to Bayesians any more.
\end{example}
\begin{example}{\bf [Savage-Dickey for the Normal Location Family]}\label{ex:normal}
Let $\{P_{\theta}: \theta \in \reals\}$ represent the normal location family, where $P_{\theta}$ with density $p_{\theta}$ has mean $\mean$ and variance $1$. We take as prior a normal distribution with mean $\mean_0$ and variance $\rho^2 > 0$, and define the precision $\lambda := \rho^{-2}$. Suppose we observe  $Y = x^n = (x_1, \ldots, x_n)$.  By standard calculations, the standard Bayesian posterior is  given by a normal distribution with mean $\breve{\mean} = (n/(n+ \lambda))\hat\mean + (\lambda/(n+ \lambda)) \mean_0$, with $\hat\mean = \hat\mean(y)$ the MLE  $(\sum_{i=1}^n x_i)/n$, and posterior variance $1/(n+\lambda)$, i.e.\ with density
\begin{equation}
    \label{eq:posterior}
w(\mean \mid x^n) = \sqrt{\frac{n+\lambda}{{2 \pi}}} \cdot e^{-\frac{(n+\lambda) (\mean - \breve\mean)^2}{2}}
\end{equation}
so that, by (\ref{eq:sd}),
\begin{equation}\label{eq:itsnormal}
\bar{P}^{\sd}(\mean \mid y) = 
\sqrt{\frac{n+\lambda}{\lambda}} \cdot e^{- \frac{n+\lambda}{2} (\mean - \breve\mean)^2 + \frac{\lambda}{2} \cdot (\mean - \mean_0)^2}
 = 
\sqrt{\frac{n+\lambda}{\lambda}} \cdot e^{- \frac{n}{2} (\mean - \hat\mean)^2 + \frac{1}{2} \cdot \frac{n \lambda}{n+\lambda} \cdot (\hat\mean - \mean_0)^2}
\end{equation}
where the latter equality follows by simple calculus when $\mean_0 = 0$ and reducing the general case to  this case by considering translated data $x_1- \mean_0, \ldots, x_n - \mean_0$; see
Figure~\ref{fig:eposterior}. 
\end{example}
\begin{example}{\bf [Two-Point Prior-Based E-Posterior for Normal Location]}
\label{ex:discrete}
Fix some $C> 0$ and set $n^*$ equal to the sample size $n$ (in Example~\ref{ex:normalcontinued} we generalize to and explain cases with $n^* \neq n$) and, for each $\theta$, define $\theta^- < \theta $ and $\theta^+ > \theta$ such that
\begin{equation}\label{eq:cinzia}
  \frac{1}{2} 
(\mean - \meanplus)^2 = \frac{1}{2}
 (\mean - \meanmin)^2 = \frac{C}{n^*}.
\end{equation}
Then $S_{\mean}(y) =  \frac{(1/2) p_{\meanmin}(y)+ (1/2) p_{\meanplus}(y)}{p_{\mean}(y)}$ is  an e-variable (just plug it into the definition to check). It is not of Savage-Dickey form, since $\meanmin$ and $\meanplus$ depend on $\theta$. In fact it coincides with the two-sided version of Johnson's (\citeyear{johnson2013uniformly}) {\em uniformly most powerful Bayes factor\/}  at level $\alpha$ if we were to set $C:= - \log \alpha$, 
which might, according to Johnson, motivate its use for testing $\cH_0 = \{P_{\theta} \}$ against $\cH_1 = \{P_{\theta'}: \theta' \in \Theta, \theta'\neq \theta \}$. Here we will use it for estimating and not just testing, i.e.\ for general $\theta \in \Theta$; then its Bayesian interpretation is not so clear, but it will tend to give better risk assessments (\ref{eq:riskassessment}) than the Savage-Dickey e-variable above, as we explain in Section~\ref{sec:taxonomy}. 
Both Example~\ref{ex:normal} and~\ref{ex:discrete} are extended to general regular \citep{BarndorffNielsen78} 1-dimensional exponential families in the Supplementary Material.
\end{example}
\begin{example}{\bf [Simple Hypothesis Testing with LR E-Variables]}\label{ex:simpletesting}
Consider testing a simple $\cH_0 = \{P_0 \}$ vs.\ a simple $\cH_1 = \{P_1\}$ (we remark on the composite case in Section~\ref{sec:discussion}). 
We can embed this in our setting by taking $\cP= \cH_0 \cup \cH_1$ and  as our loss function a {\em Wald-Neyman-Pearson}-type loss given by $\aspace= \cA= \{0,1\}$, $L_{\currz}: \{0,1\}^2 \rightarrow \reals^+_0$ nonnegative, and $L_{\currz}(0,1) > 0$, $L_{\currz}(1,0) > 0$. We further say that the loss function satisfies {\em Condition Zero}, henceforth abbreviated to ${\bf C0}$, if 
$L_{\currz}(0,0) = L_{\currz}(1,1) = 0$. 
The interpretation is that action $0$ stands for `accept the null' leading to loss $0$ if the true state of nature is $P_0$; and $1$ stands for `reject the null', leading to loss $0$ if the alternative $P_1$ is true. The use of such loss functions is implicit in the Neyman-Pearson testing paradigm, as noted in 1939 by  \cite{Wald39} who stated that one may `of course' set $L_{\currz}(0,0) = L_{\currz}(1,1) = 0$, i.e.\ impose {\bf C0}, and because it simplifies treatment we will  do so in most testing examples in this paper --- but in Example~\ref{ex:science} we will deviate from it, and see that this in a sense makes for the most interesting applications of our risk assessments. 
We associate both $P_0$ and $P_1$ with an e-variable. A natural choice is to take the likelihood ratio, $S_1 = p_0(Y)/p_1(Y)$ and $S_0 = p_1(Y)/p_0(Y)$, seen to be e-variables by noting ${\bf E}_{Y\sim P_{\theta}}(p_{\theta'}(Y)/p_\theta(Y)) = \int p_{\theta'}(y) dy = 1$. In fact, they are the `prototypical'  e-variables for simple-vs.-simple testing \citep{GrunwaldHK19}. The risk assessment (\ref{eq:riskassessment}) now becomes
\begin{equation}
\label{eq:lrassessment}
\bar{R}(\delta) =
\max \left\{ \frac{p_1(Y)}{p_0(Y)} L_{\currz}(1,\dr{B}(Y)),
 \frac{p_0(Y)}{p_1(Y)} L_{\currz}(0,\dr{B}(Y)) \right\}  
 \overset{\text{under  {\bf C0}}}{=}   \frac{p_1(Y)}{p_0(Y)} L_{\currz}(1,\dr{B}(Y))+
 \frac{p_0(Y)}{p_1(Y)} L_{\currz}(0,\dr{B}(Y)).
\end{equation}
\end{example}
\begin{example}
{\bf [Simple Testing with BBW E-Variables]}
\label{ex:bbwfirst}
Several other e-variables may be considered for simple hypothesis testing as well, including the Savage-Dickey e-variables; we treat these in Section~\ref{sec:taxonomy}. There we also show that BBW's  conditional error probabilities \citep{BergerBW94,Wolpert96b,Berger03} can be interpreted as e-posteriors. 
Formally, we obtain a parameterized set of e-posteriors \mbox{$\{ \bar{P}^{\bbw}_w(\theta \mid y) \}$} for $0 < w < 1$. We consider the special case $w=1/2$ here, deferring the general case and definitions to Section~\ref{sec:conditionalistfrequentist}. As we will see there, this e-posterior is quite special:  
for the e-posterior minimax  decision rule $\deltamm$ we get,  with $\bar{R}$ defined relative to the BBW e-posterior, under ${\bf C0}$,
\begin{eqnarray}
\label{eq:kbbwassessment}
\bar{R}(\deltamm) & {=} &  
\frac{p_1(Y)}{p_0(Y)+ p_1(Y)} L_{\currz}(1,\deltamm(Y))+
 \frac{p_0(Y)}{p_0(Y)+p_1(Y)} L_{\currz}(0,\deltamm(Y)). 
\end{eqnarray}
Thus, it can again be rewritten as a sum rather than a supremum and we see that it is formally identical to the Bayes assessment  (\ref{eq:bayesassessment}) when $P_0$ and $P_1$ are equipped with the uniform prior. We thus have a {\em frequentist\/} risk assessment, not requiring a prior assumption on $\cP$,  that is as strong as a Bayesian one, which is optimal under a specific prior assumption. 
%
%
This seems almost too good to be true, especially since it is slightly, but uniformly, better than the alternative (\ref{eq:lrassessment}) --- and as we will see in Section~\ref{sec:conditionalistfrequentist}, there is indeed no free lunch: (\ref{eq:kbbwassessment}) only holds for $\deltamm$, which for the BBW e-variable can sometimes be a rather bad decision rule, whereas (\ref{eq:lrassessment}) holds for arbitrary rules; there are some other prices to pay as well when using BBW's e-variable. In fact, a major motivation for the author's involvement in e-variables was my realization, some fifteen years ago, that when deriving frequentist Type-I error bounds for decisions based on Bayes factors/likelihood ratios in simple testing, one ends up with the expression $p_1/p_0$ as in (\ref{eq:lrassessment}) if one follows the reasoning of classic Bayesian texts such as \citep{EdwardsLS63,Good91} or Royall's {\em universal bound\/}
\citep{royall1997statistical}; whereas using BBW's ideas, one obtains $p_1/(p_1+p_0)$ as in  (\ref{eq:kbbwassessment}): the small but essential discrepancy between the two assessments  tripped me up for about 15 years until I finally developed the necessary tool to understand it --- the much more general e-posterior of which both approaches are simply a special case, each with its own (dis)advantages,  as discussed in Section~\ref{sec:conditionalistfrequentist}.
\end{example}
\begin{figure}
    \centering
    \includegraphics[width=0.3\textwidth]{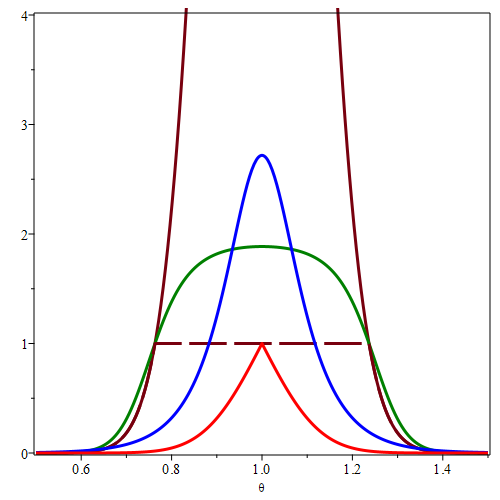}  \ 
    \includegraphics[width=0.3\textwidth]{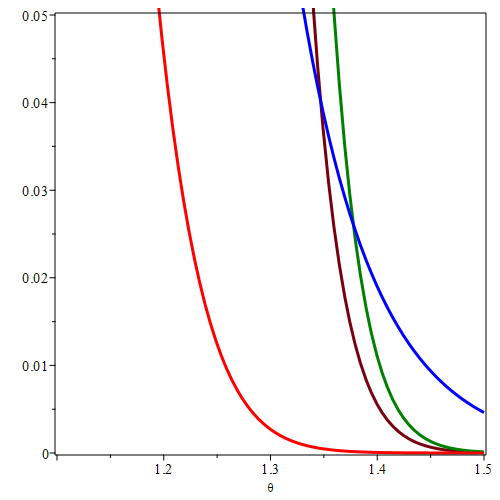} \
    \includegraphics[width=0.3\textwidth]{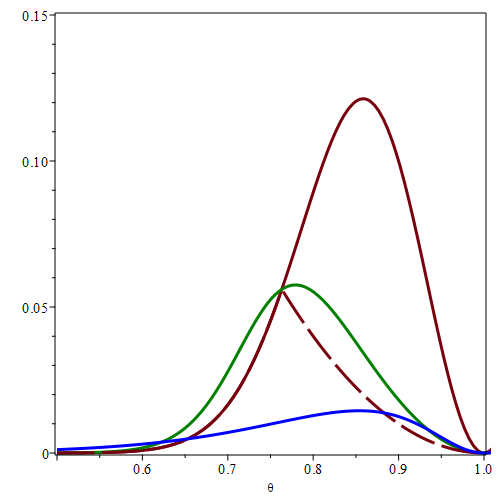} 
    \caption{
    {(\em left and middle)\/} The brown solid curve is the pure Savage-Dickey e-posterior $\bar{P}(\theta \mid y)$  for the normal location family based on  sample $y= (x_1, \ldots, x_{100})$ with $\hat\mean(y) = 1$ and prior $W$ with $\lambda =1$ and $\theta_0=0$ as in (\ref{eq:itsnormal})  (it reaches a maximum of about $17$). The green curve represents the $1/2$-dampened version (see (\ref{eq:dampen})) of this e-posterior, by design  bounded by $2$. The dashed brown curve is the  capped version $\capped{P}(\theta \mid y)$ (Definition~\ref{def:capped}) of $\bar{P}(\theta \mid y)$, equal to it whenever the latter is $\leq 1$. The blue curve represents the generalized Savage-Dickey two-point prior-based e-posterior optimized for $n^*=n=100$ as in Example~\ref{ex:discrete} and~\ref{ex:normalcontinued}.  The red curve is twice the tail area of the objective Bayes posterior $w^{\circ}(\theta \mid y)$ of Example~\ref{ex:gaussexample}, i.e.\ given by $f(\theta) =2 W^{\circ}(\bar\theta \geq \theta \mid y) = 2 \int_{\theta}^{\infty} w^{\circ}(\theta \mid y) d \theta$.
    Viz. the remark underneath Definition~\ref{def:capped}, the set $\{\theta: \bar{P}(\theta \mid y) \geq \alpha \}$ is a conservative 2-sided $(1-\alpha)$-confidence interval, whereas, viz. Example~\ref{ex:gaussexample},  the set   $\{\theta: f(\theta) \geq \alpha \}$ is an exact nonconservative 2-sided $(1-\alpha)$-confidence interval, so that plotting $f(\theta)$ provides a better comparison to $\bar{P}(\theta \mid y)$ than directly plotting $w^{\circ}(\theta \mid y)$. The middle display zooms in on the $0.05$ right tail. \\ {(\em right)\/} The value of $(\theta -\hat\theta)^2 \bar{P}(\theta \mid y)$ is depicted for the same four e-posteriors, using the same colors. The two-point prior-based e-posterior is optimized to give good assessments relative to squared error, its maximum is considerably smaller than for the others; for the pure Savage-Dickey prior, the maximum is excessively high. 
    \label{fig:eposterior}} 
\end{figure}
\section{Quasi-Conditional Risk Assessment}
\label{sec:riskassessment}
We  want risk bounds that are valid, in some precise sense, simultaneously for all  loss functions of the form $L: \Theta \times \cA \rightarrow \reals$ that satisfy the no-sure-gain condition (see underneath (\ref{eq:riskassessment})). Proposition~\ref{prop:TheBound} below shows that, for any choice of e-variable collection $\cS =  \{S_{\theta}: \theta \in \Theta \}$ and any given decision rule $\delta$, we have that  $\bar{R}(\delta)$ as in (\ref{eq:riskassessment}) satisfies this purpose.
\begin{proposition}\label{prop:TheBound}{\bf [Validity of E-Risk Assessments]}
Consider a decision problem as defined in the introduction relative to any model $\{P_{\theta}: \theta \in \Theta\}$ for random variable $Y \in \cY$ and loss function $L$ as above that satisfies the no-sure-gain condition.  
Let $\bar{R}(\delta)$ be as in (\ref{eq:riskassessment}) relative to some e-variable collection $\cS= \{S_{\theta}: \theta \in \Theta\}$.
For all $\theta \in \Theta$ we have: 
\begin{equation}\label{eq:TheBoundA}
{\bf E}_{Y \sim P_{\theta}} 
\left[
\frac{L_{\currz}(\theta,\dr{B}(Y))}{\bar{R}{(\delta)}}
\right]  \leq 1,
\end{equation}
where we adopt the convention $0/0 :=0$. In particular, this implies, for every function 
$b: \cY \times \cU \times \Theta \times \cA \rightarrow \reals^+$, and for any random variable $U$ taking values in some set $\cU$ and any distribution $P'_{\theta}$ for $(Y,U)$ whose marginal on $Y$ coincides with $P_{\theta}$, that  
\begin{equation}\label{eq:TheBoundB}
{\bf E}_{(Y,U)\sim P'_{\theta}} 
\left[
{b(Y,U,\theta,\delta(Y)) \cdot L_{\currz}(\theta,\dr{B}(Y))}
\right]  \leq
\sup_{(y,u,\theta) \in \cY \times \cU \times \Theta}  {b(y,u,\theta,\delta(y))} \cdot {\bar{R}(y,\dr{b}(y))}.
\end{equation}
Moreover,
for any e-collection $\cS$,
the  risk  bound  (\ref{eq:riskassessment}) with any minimax decision rule $\deltamm$ as given by (\ref{eq:minimax}) {\em dominates\/}  (\ref{eq:riskassessment}) when applied with any other decision rule $\delta$:
for all $y \in \cY$: $\bar{R}(y,\drmm{b}(y)) \leq \bar{R}(y,\dr{b}(y))$.
\end{proposition}
\begin{proof}
Fix any $\theta \in \Theta$. By definition of $\bar{R}({y,a})$, we have for all $y \in \cY$ that 
$\bar{R}({y,\dr{b}(y)}) \geq \bar{P}(\theta \mid y) L_{\currz}(\theta,\dr{b}(y))$ so that 
\begin{align}
{\bf E}_{Y \sim P_{\theta}}  \left[
\frac{L_{\currz}(\theta,\dr{B}(Y))}{\bar{R}(\delta)}
\right]  \leq  
 {\bf E}_{Y \sim P_{\theta}}  \left[ 
\frac{L_{\currz}(\theta,\dr{b}(Y))}{
 S^{-1}_{\theta}(Y) L_{\currz}(\theta,\dr{b}(Y))
}
\right]  = {\bf E}_{Y \sim P_{\theta}} \left[ S_{\theta}\right] 
%
\leq 1,
\end{align}
noting that the derivation is valid as long as the denominator $\bar{R}(\delta)$ inside the first expectation cannot evaluate to $0$. In case it can, we split up the expectand into two terms, ${\bf 1}_{\bar{R}(\delta) = 0} \cdot {L_{\currz}(\theta,\dr{B}(Y))}/{\bar{R}(\delta)}+ {\bf 1}_{\bar{R}(\delta) > 0} \cdot {L_{\currz}(\theta,\dr{B}(Y))}/{\bar{R}(\delta)}$, and use that the first fraction is $0/0$ which by our convention evaluates to $0$; we proceed with the second term as above.
This proves (\ref{eq:TheBoundA}). (\ref{eq:TheBoundB}) and domination of $\drmm{b}$ is then immediate.
\end{proof}
The bounds (\ref{eq:TheBoundA}) and (\ref{eq:TheBoundB}) (we mostly use the latter) are unusual and for them to be practically interesting, they should satisfy several desiderata. In particular: 
\begin{enumerate}
    \item 
They should have a clear interpretation. We show that this is the case directly below.
\item They should exceed standard risk bounds obtainable for standard estimators (actions) in standard decision problems by only a small factor, at least if we choose a reasonable e-collection $\cS$. This is indeed the case in our running examples (see Example~\ref{ex:normalcontinued},  for the Gaussian location family with standard squared error loss, Example~\ref{ex:normalzero} for an extension to exponential families, and  Example~\ref{ex:testinggsd}, underneath (\ref{eq:twohypothesessdb})  for simple-vs.-simple testing).  
\item There should exist settings in which, just like Bayesian risk assessments, they 
allow one to make more risky decisions conditional on more extreme data. This happens in Example~\ref{ex:science} below. 
\item There should exist situations in which they give useful bounds whereas a standard  Bayesian assessment, based on a standard convenience prior, is really highly misleading. We give such a case in Example~\ref{ex:gaussexample}. 
\end{enumerate}
\paragraph{Bookie Interpretation}
The following interpretation of (\ref{eq:TheBoundB}) has a {\em De Finettian\/} flavour to it. Suppose a bookie offers decision-maker (DM) to play the following game, really a meta-decision problem: DM first specifies a  decision-rule $\delta$ for the original decision-problem. DM then gets paid some amount $\ell > 0$; then $\theta$ and $Y$  are revealed to both DM and bookie. Then bookie presents a collection of numbers $\{B_{\theta,a}: \theta \in \Theta, a \in \cA \}$, all assumed  $> 0$. Bookie may choose $B_{\theta,a}$ as functions of the data $Y$ and external data, here encoded as $U$; hence we have $B_{\theta,a}=b(Y,U,\theta,a)$ for some function $b$.  DM then has to pay back $B_{\theta,\delta(Y)} \cdot  L(\theta,\delta(Y))$. {\em If the rules are such that (i.e. bookie a priori guarantees that), no matter what $Y,U$ are observed, bookie will provide a $B$ such that  $B_{\theta,\delta(Y)} \leq \ell/\bar{R}(\delta)$ then DM should accept to play this game, for her expected gain is $\geq 0$; if this cannot be guaranteed, then DM has no guarantee that the expected gain is nonnegative.}
This directly follows from (\ref{eq:TheBoundB}); note there is no need for DM to know the details of how $B=b(Y,U,\theta, \delta(Y))$ is defined, as long as the promise $B_{\theta,\delta(Y)} \leq \bar{R}^{-1}(\delta)$ always holds. 

In all examples to follow, we consider a {\em simple\/} bookie, by which we mean that the $B=b(y,u)$ presented is a single number that can be written as a function of $Y$ and $U$ only, and does not vary with  $\theta$ or $a$. The effect of this is that, both for a Bayesian DM and for an e-posterior minimax DM,  knowledge of $B$ does not give any incentive to change the action taken: for any given prior $W$, a decision rule $\delta$ is Bayes relative to loss function $L$ iff it is Bayes relative to loss function $L'(\theta,a) := b(y,u) \cdot L(\theta,a)$; and the same holds for e-posterior minimaxity. 
Bookies who provide a $B_{\theta,a}$ depending on $\theta$ and $a$ are certainly worth looking at as well, but the analysis is substantially more difficult and will be left for future work. 
\begin{example}{\bf [Policy Makers vs. the Decision Maker Interpretation]}
\label{ex:science}
While the bookie interpretation may seem rather abstract, here we show that it can translate into practically  relevant  settings in which  risk bound  (\ref{eq:TheBoundB}) has actual decision-theoretic consequences, allowing  one to reliably make more extreme decisions if the data are more extreme. In standard hypothesis testing problems, this cannot happen, since there are only two actions,  `reject' and `accept', and then `more extreme actions'  do not exist. But in practical testing problems there  is often an additional factor, measuring the {\em importance\/} of the decision in light of the observed data. 
%
The following real-world example  (in essence taken from \cite{Grunwald22}) illustrates:  a study about vaccine efficacy ({\sc ve}) in a pandemic has been set up as a test between null hypothesis $\text{\sc ve} \leq 30\%$ and alternative $\text{\sc ve} \geq 50\%$. The original plan was to vaccinate all people above 60 years of age if the null is rejected. But suppose the null  actually gets rejected with unexpectedly strong evidence for the alternative --- there is substantially more evidence than policy makers had hoped for ---  and at the same time the virus' reproduction rate may be much higher than anticipated. Based on  the observed data and the changed circumstances, the policy makers might now contemplate an alternative action: vaccinate every adult, instead of everyone above 60. 
This would change the losses of all decisions. We can model this (in an admittedly highly stylized way) by assuming that the loss incurred is really  $B \cdot L(\theta,a)$ where $B$ is some positive number depending on how the decision `$\delta(y) = a$' really gets translated into practical measures. For example, $\delta(y)=1$ (vaccine effective) may translate into `vaccinate everyone over the age of $V$'. If  $L(1,1)$ and $L(0,0)$ are both negative (as we allow),  then the lower $V$, the higher the value of $B$ will be.  
As indicated above, we simplify things by assuming that the  $B = b(y,u)$ ultimately chosen does not depend on $\theta$ or the action taken; all losses are multiplied by the same amount. 
In practice, the employed weight  $B$ might arise from a government committee (say the `policy makers') that analyzes the vaccine study results independently from DM, potentially in a very different way (e.g. fully Bayesian, or just informally by eyeballing graphs). The more convincing they determine the evidence to be, the higher the $B$  they may propose; similarly, $B$ may depend on external  $U$ --- if the question is more pressing then one might want to take more drastic measures so the $B$'s go up.  The role of the DM (who should now perhaps be called the `final risk checker') is simply to formulate an absolute upper bound $\ell$ for the risk and tell the policy makers: 
as long as you guarantee policies 
so that $B \cdot \bar{R}(Y,\delta(Y)) \leq \ell$, I can guarantee you a risk of no more than $\ell$ (here $\ell$ should be determined before the data become available; in contrast, the details of how $B$ arises need not be known and can be arrived at by `vague' means, as long as $B$ is conditionally (given the data) independence of $\theta$ and $a$). 
\end{example}
\paragraph{The Quasi-Conditional Stance}
By stressing the importance of (\ref{eq:TheBoundB}), we adopt a philosophy  closely related to Neyman's (\citeyear{Neyman50}) influential {\em inductive behaviour\/} view of statistics --- we cannot say anything about our performance for  an individual study, but we can make claims about our expected performance (or, by the law of large numbers and concentration inequalities, with high probability, about our average performance in several runs).
One can say less than under strict Bayesian assumptions, i.e the assumption that one fully believes one's prior: in that case, the assessments one makes will be conditionally correct, and not just in the average sense of (\ref{eq:TheBoundB}). But there is still a `conditional' aspect to (\ref{eq:TheBoundB}), in the informal sense that more extreme data (small $\bar{R}(\delta)$) leads to a decision rule being acceptable at higher `stakes' (value of $B$). More formally, upon observing $\bar{R}(\delta) = r$, a Bayesian who insists on risk bounded by $\ell$ would condition on this event and hence accept importance weight $B$ iff $B \cdot r \leq \ell$ --- a simple bookie can present any $B$ he likes and the Bayesian can decide based on $r$ whether to accept or not. 
Our DM can safely accept $B$ if the subtly stronger property holds that  it is known {\em a priori\/} that the bookie must present a $B$ with $B \cdot  \bar{R}(\delta)\leq \ell$, no matter the value $\bar{R}(\delta)$ takes, or equivalently, DM can tell, after observing $\bar{R}(\delta) = r$, bookie that he must present a $B \leq \ell/r$ and bookie is guaranteed to obey. Thus, our DM can safely accept higher stakes if better bounds are observed and is thus more flexible than a DM in the unconditional frequentist Neyman-Pearson setting; but only a subset of the games that are acceptable to the conditioning Bayesian is acceptable to the DM --- in this sense, she is `quasi-conditioning'. Such substantial consequences ensuing from subtle differences in the protocol between bookie and DM are also found, in a different context, in \citep{GrunwaldH11}, and are a subject matter of the Vovk-Shafer theory of {\em game-theoretic probability\/} \citep{ShaferV19}.
They motivate the  term {\em quasi-conditional}, emphasizing that the DM may behave in some but not all senses as if conditioning on the data; and that her final performance 
is evaluated in expectation over all possible data, and not conditionally on any given data.

To our knowledge, this is a new paradigm. \commentout{
The classical Neyman-Pearson approach and general frequentist decision theory typically provides methods for a specific, a priori given loss function (e.g. minimax estimators) or decision problem (e.g. confidence distributions for valid confidence intervals, see below) and cannot deal with data-dependent weights $b(y)$ --- which would e.g. be akin to setting  a confidence level $1-\alpha$ with $\alpha$ depending on the data \citep{Grunwald22}. PAC-Bayesian approaches likewise provide algorithms that only work for a fixed, a priori given loss function. } 
It is certainly different from the main existing attempts to unify Bayesian ideas (represent posterior uncertainty by a data-dependent distribution over parameters) and frequentist guarantees  --- here we think of (a) objective Bayes methods with {\em matching\/} priors chosen to get frequentist coverage \citep{BergerBS22}; (b) confidence and the related {\em fiducial\/} distributions \citep{SchwederH16}; and (c) conditionalist frequentist methods such as BBW's. The example below illustrates this for (a) and (b), focusing on the normal location family, for which matching objective Bayes priors, standard objective Bayes priors, confidence distribution and fiducial distribution all agree. It shows that --- despite the matching! --- if one were to use them as if they were standard posteriors in decision problems with an incentive to play more extreme actions with more extreme data, they fail to give proper frequentist error bounds.
(c) is more subtle: in Section~\ref{sec:conditionalistfrequentist} we show that  BBW and other conditionalist frequentist approaches can be re-interpreted in terms of e-posteriors --- so they can be used against arbitrary loss functions even though there were not designed as such. However, these e-posteriors only give meaningful $(< \infty)$ bounds for a small subset of all decision rules (in the BBW case, effectively only for a single one (!)), making them highly restricted in practice  --- whereas other e-collections such as the Savage-Dickey and LR ones give meaningful bounds for essentially any decision rule. 
\begin{example}{\bf [Weighted Squared Error]}
\label{ex:gaussexample}
Take the Gaussian location family with variance 1 as in Example~\ref{ex:normal} and~\ref{ex:discrete}, where we allow data-dependent stopping times such as `$\tau$ is the smallest $n$ at which $x_n > 5$'), defined in the standard manner \citep{Williams91,hendriksen2021optional}. Hence, data $Y$ takes the form $Y= (\tau,X^\tau)$ with $\tau$ a stopping time and $X^\tau$ a vector of $\tau$ outcomes.  Suppose we aim to find the estimator with smallest mean square error, where the importance of the problem at hand can depend on the data  via some function $B = b(Y)$ (we will ignore external circumstances $U$); it is known to DM that $B$ is some deterministic function of $Y$. 
We thus set
\begin{equation}\label{eq:squarederror}
L(\theta,a) =  (\theta- a)^2,
L_{\curry}(\theta,a) := b(y) L(\theta,a) 
\end{equation}
where $L_{\curry}$ is the loss function we are interested in in the end, 
a squared error loss with importance weighted by some function $b(y)$. In general, $B=b(Y)$ may again be determined by policy makers  who may reconsider, using potentially vague ways, the importance of the problem in light of the observed data as in Example~\ref{ex:science} above. We can think of any estimator $\breve\theta$ for $\theta$ as a decision rule, that, upon observing $Y= (n,X^n$), outputs $\breve\theta:= \breve\theta(n,X^n)$. 

We first consider how one would go about this problem in an `objective Bayes' manner. The standard (uniform, improper) `objective Bayes' prior for this family and data $x^n$ corresponds to the limit for $\lambda \downarrow 0$ in (\ref{eq:posterior}). The ensuing posterior $W^{\circ} \mid x^n$ has a normal density $w^{\circ}(\theta \mid x^n)$ with mean and median equal to the maximum likelihood estimator (MLE) $\hat\theta(x^n) = n^{-1} \sum_{i=1}^n x_i$
and variance $1/n$, and it is `matching' \citep{BergerBS22} in the sense that $1-\alpha$ credible posterior intervals (taken symmetrically around the posterior mean) coincide exactly, for each fixed $n$, with $1-\alpha$ confidence intervals. In this case the objective Bayes posterior also coincides with the {\em fiducial\/} and the {\em confidence\/} distribution  \citep{SchwederH16} based on $x^n$. 
The name `confidence distribution' already warns that this posterior should perhaps only be used for `confidence' (statements about confidence intervals) but in the fiducial and objective Bayes world, it seems to have been advocated more generally. That this is problematic was already exhibited in different settings  by \cite[Example 7]{grunwald2018safe} and \cite{balch2019satellite}.
As is well-known, Bayesian posteriors do not depend on the definition of the stopping time \citep{BergerW88,hendriksen2021optional}; also, they condition on $Y= (n,X^n)$ so further conditioning on $B = b(Y)$ also does not affect them, since it is completely determined by $Y$. Thus we have $w^{\circ}(\theta \mid Y = (n, X^n),B=b_0)= w^{\circ}(\theta \mid x^n)$ \citep{hendriksen2021optional}
so that the Bayes optimal action based on (\ref{eq:squarederror}) and $W^{\circ}(\theta \mid Y= (n,x^n),B=b_0)$ is simply the MLE, $\hat\theta(x^n)$ irrespective of the definition of $\tau$, the definition of $b$ and the value of $b(y)$. Now, based on this posterior,
the loss we {\em think\/} we make
upon observing $Y=(n,x^n)$ is given, as in (\ref{eq:bayesdecision}), by
\begin{equation}\label{eq:obayesassessment}
\bar{R}^{\textsc{bayes}}(W^{\circ} \mid n,x^n) = {\bf E}_{\bar{\theta} \sim W^{\circ} |x^n} [L_{\curry}(\bar\theta,\hat\theta)]
= \frac{b(n,x^n)}{n},\end{equation}
where we used that $W^{\circ} \mid Y$ is a normal with mean $\hat\theta$ and variance $1/n$, independently of the definition of the stopping time random variable $\tau$.  
Therefore, if data are in fact sampled from $\theta^*$, then the  average of the losses we expect to make, in  several studies, is given by 
\begin{equation}\label{eq:above}
 {\bf E}_{(\tau,X^\tau) \sim P_{\theta^*}}[{\bf E}_{\bar\theta \sim W^{\circ} \mid X^\tau}  [L_{\currY}(\bar{\theta},\hat\theta(X^\tau))]] = 
   {\bf E}_{(\tau,X^\tau) \sim P_{\theta^*}}\left[ \frac{b(\tau,X^\tau)}{\tau}\right],
 \end{equation}
whereas the loss we {\em actually\/}  make on average is 
\begin{equation}\label{eq:rstar}
\bar{R}^* := {\bf E}_{(\tau,X^{\tau}) \sim P_{\theta^*}}[L_{\currY}(\theta^*,\hat\theta(X^{\tau}))] = 
{\bf E}_{(\tau,X^\tau) \sim P_{\theta^*}}[
b(\tau,X^\tau) \cdot (\theta^* - \hat\theta(X^\tau))^2 
]
\end{equation}
We now show that the difference between believed risk (\ref{eq:above}) and actual (\ref{eq:rstar}) can be unbounded, i.e. we can get it to be at least $k^*$, for arbitrarily large $k^*> 0 $. To this end, simply set $b(\tau,X^\tau):= \tau$  and set the stopping time $\tau$ to be the smallest $n$ such that $(\hat\theta(x^n) - \theta^*)^2/n \geq k^*$. By the law of the iterated logarithm \citep{Williams91} we know that $\tau$ is finite, $P_{\theta^*}$-almost surely, under any $\theta^* \in \Theta$.
With this definition of $b$, the believed risk  (\ref{eq:above}) is $1$ yet the actual risk (\ref{eq:rstar}) is at least $k^*$. By Proposition~\ref{prop:TheBound} such a discrepancy is impossible if we base our risk assessment on e-posteriors via (\ref{eq:riskassessment}) instead. This will be made concrete in Example~\ref{ex:normalcontinued} and~\ref{ex:logn} below.  
\end{example}
\section{A Small Taxonomy of E-Specifications}
\label{sec:taxonomy}
Existing papers on e-variables such as \cite{TurnerLG21,perez2022estatistics,wasserman2020universal} consider $P_{\theta}$ that define a random process $X_1, X_2, \ldots$, and then they provide constructions for an e-variable for sample size $n$, for each $n$. Thus, they really provide e-{\em specifications\/}: a sequence of e-variables $S^{[1]}_{\theta}, S^{[2]}_{\theta}, \ldots$, for each $\theta \in \Theta$, where $S^{[n]}_{\theta}$ is an e-variable for data $Y= (X_1,\ldots, X_n)$. 

Straightforwardly, upon defining $q(X^n) = p_{\theta}(X^n) S^{[n]}_{\theta}(X^n)$, we have $\int q(x^n) d \mu(x^n) \leq 1$ where $\mu$ is the underlying measure, so $q$ can be thought of as a probability density (if $\int q d \mu < 1$, we may think of the missing mass being put on an outcome that, under $P_{\theta}$, has probability $0$). $S_{\theta}^{[n]}$ can then be written as a likelihood ratio $q(X^n)/p_{\theta}(X^n)$. Thus, in the {\em simple\/} (no-nuisance) parameter case, {\em we can rewrite every e-variable as a likelihood ratio}. Importantly though, $q$ is allowed to depend on $n$ and $\theta$; the specification $S_{\theta}^{[1]},S_{\theta}^{[2]},\ldots$ may deliver, for different $n$ and $\theta$, entirely different $q$'s. To emphasize this possibility, we will from now on make the potential dependence on $n$ and $\theta$ explicit, and henceforth write
\begin{align}\label{eq:eprocesses}
    S^{[n]}_{\theta} = \frac{q^{[n]}_{\leftlsquigarrow \theta}(X^n)}{p_{\theta}(X^n)}. 
\end{align}
For an important subclass of  e-variable specifications, the corresponding $q^{[n]}_{\leftlsquigarrow \theta} = q_{\leftlsquigarrow \theta}$ does not depend on $n$ after all (it may still depend on $\theta$). In that case $q_{\leftlsquigarrow \theta}$ defines a random process for $X_1, X_2, \ldots$ and the specification $(S^{[n]}_{\theta})_n$ is called an {\em e-process\/} \citep{ramdas2022testing}. E-processes have a major advantage: for every e-process, for every stopping time $\tau$ (e.g.\ `stop at the smallest $n$ at which $S^{[n]} \geq 1/\alpha$, and set $\tau$ to this $n$'), we have \citep{ramdas2022testing}:
$$
{\bf E}_{X^{\tau} \sim P_{\theta}}\left[ S^{[\tau]} (X^{\tau})\right]
=
{\bf E}_{X^{\tau} \sim P_{\theta}}\left[ \frac{q_{\leftlsquigarrow \theta}(X^{\tau})}{p_{\theta}(X^{\tau})}\right] \leq 1. 
$$
As a result, with e-processes we can engage in {\em optional stopping\/} --- under any stopping time $\tau$ --- even the aggressive stopping rule mentioned above or some externally imposed rule, the details of the definition of which we do not need to know --- $S^{[\tau]}_{\theta}$ is still an e-variable, and decision making based on (\ref{eq:riskassessment}) as above can proceed. With e-specifications that do not provide e-processes, one cannot engage in such optional stopping  (see the extended discussion by \cite[Section 5]{GrunwaldHK19}). 
Indeed, in Example~\ref{ex:gaussexample}, which involved a non-constant stopping time $\tau$, one only gets valid assessments with e-processes rather than mere e-variables, and we shall use these in Example~\ref{ex:normalcontinued} and~\ref{ex:logn} to provide risk assessments for that setting. 

We will now reconsider the e-collections of all our running examples --- as we will see, some provide e-processes, others do not. 
\subsection{Generalized Savage-Dickey E-Processes}
\label{sec:generalizedsd}
We have already seen two instances of e-processes: the (specification, for each $n$, of) Savage-Dickey e-variables $S^{\sd}$ and the likelihood ratio e-variables $S^{\lr}$. Both are instances of what we may call {\em generalized Savage-Dickey e-processes\/} $(S^{[n]})_{\theta}$, with
\begin{align}\label{eq:SD}
    S^{[n]}_{\theta} =  \frac{p_{W \leftlsquigarrow \theta}(X^{n})}{p_{\theta}(X^{n})}
\end{align}
where $p_W(X^n) = \int p_{\theta}(X^n) dW(\theta)$ is a Bayes marginal distribution on $n$. Thus, the e-variable numerator is a Bayes marginal, but the prior $W$ one uses is allowed to depend on the parameter $\theta$ one is comparing to. 
\begin{example}{\bf [Gaussian Location, Example~\ref{ex:discrete} Continued]}\label{ex:normalcontinued}
Fix some {\em anticipated\/} sample size $n^*$, representing our best a priori guess of the actual sample size $n$ that we will observe. Our aim is to obtain an e-process that is valid for all $n$, including $n \neq n^*$, but is optimized to give the best possible risk bound if we happen to observe $y =x^n$ with $n$ equal or close to $n^*$. To this end, for each $\mean \in \reals$ we define $\meanmin < \mean$ and $\meanplus > \mean$ to satisfy (\ref{eq:cinzia}) with $C=1$.
The construction below gives a valid e-process for every $C > 0$,  but $C=1$ gives the best bounds for the squared error, at the cost of potentially somewhat worse yet of course still valid, bounds for other loss functions. 
Now define, for each $n$, the e-variable $S^{[n]}_{\mean}(y) =  \frac{(1/2) p_{\meanmin}(y)+ (1/2) p_{\meanplus}(y)}{p_{\mean}(y)}$. Clearly the $S^{[n]}_{\theta}$ form a generalized Savage-Dickey e-process, with $W \leftlsquigarrow \theta$ the uniform distribution on $\{\meanmin,\meanplus\}$. 
In the Supplementary Material (Section~\ref{app:mleminimax}, Proposition~\ref{prop:KLahoy}) we show that for the corresponding e-posterior $\bar{P}(\theta \mid y)$, we have that
$$
\min_a  \max_{\theta \in \reals}  (\theta- a)^2 \bar{P}(\theta \mid y)
$$
is achieved by setting $a = \hat\theta$ (independently of $n$ and $n^*$), 
so that $\deltamm(y) = \hat\theta$, and we also show that (Section~\ref{app:normalproofs})  (\ref{eq:riskassessment}) evaluates to \begin{equation}\label{eq:superbound}
\max_{\theta \in \reals} (\theta- \hat\theta)^2 \bar{P}(\theta \mid y)
= 
\frac{1}{n} \cdot g(n,n^*) \cdot 1.45 \ldots 
\overset{\text{if\ } n^*=n}{=}\frac{1}{n} \cdot 1.45\ldots
\end{equation}
where $g(n,n^*) = 
\left(\frac{n^*}{n}\right) \cdot  e^{n/n^*-1}$. Note that $g(\cdot,n^*)$ has a minimum at $n=n^*$ with  $g(n,n^*) =1$ and increases very slowly for $n/n^*$ close to 1. The maximum over $\theta$ is, for $n =n^*$, 
achieved approximately at $\theta = \hat\theta \pm 1.46/\sqrt{n}$  at which $\bar{P}(\theta \mid y) \approx 0.68$, independently of the value of $n=n^*$.
$S_{\theta}^{[n]}$ gives an e-process for any $n^*$ we choose in (\ref{eq:cinzia}). 
In the Supplementary Material, Section~\ref{app:normalproofs}, we show that (\ref{eq:superbound}) still holds up to lower order factors, for general 1-dimensional families, as long as they are restricted to a compact subinterval of the parameter space. 

Incidentally, this shows that for {\em standard\/} decision problems (e.g. $n$ fixed in advance, uniform weight function),  the e-posterior based risk bounds (\ref{eq:riskassessment}) are not much worse than the Bayesian risk assessment and minimax risk bounds for the same model, as long as  sufficiently `clever' e-collections are used: the factor is a modest $1.45$ if $n$ can be assumed known in advance, and a $\log n$-factor otherwise. 
\end{example}

\begin{example}{\bf [Hypothesis Testing with Generalized Savage-Dickey and LR]}\label{ex:testinggsd} Consider a simple-vs.-simple test as in Example~\ref{ex:simpletesting}. From (\ref{eq:SD}) we see that the generalized Savage-Dickey e-variable at fixed sample size $n$, with $Y=X^n$, and based on priors $W \leftlsquigarrow \theta$ with mass function $w_{\theta}$, for $\theta \in \Theta = \{0,1\}$, is given by 
\begin{align}
    S_\theta (y) = \frac{w_\theta(0) p_0(y) + w_\theta(1) p_1(y)}{p_\theta(y)} \ {\text{with}}\ w_\theta(1) = 1- w_\theta(0).
\end{align}
For degenerate priors with $w_0(0)= w_1(1)=0$ this reduces to the LR e-process of (\ref{eq:lrassessment}), and for `pure' (i.e.\ nongeneralized) Savage-Dickey e-posteriors, we have $w_0 = w_1$. 
Consider Wald-Neyman-Pearson loss functions under Condition Zero ({\bf C0}) as in Example~\ref{ex:simpletesting}.
Generalizing (\ref{eq:lrassessment}), the e-posterior minimax decision rule $\deltamm$ upon observing $Y=y$ is then given by $\drmm{b}(y) =a$, for $a$ achieving
\begin{align}\label{eq:twohypothesessd}
    & \min_{a \in \{0,1\}} \max \left\{ L_{\currz}(0,a) \cdot\frac{p_0(y)}{w_0(0) p_0(y)+ w_0(1) p_1(y)} , L_{\currz}(1,a) \cdot\frac{p_1(y)}{w_1(0) p_0(y)+ w_1(1) p_1(y)} \right\}  \nonumber \\ 
    \overset{\text{{\bf C0}}}{=} & \min_{a \in \{0,1\}} \left(  L_{\currz}(0,a) \cdot\frac{p_0(y)}{w_0(0) p_0(y)+ w_0(1) p_1(y)} +  L_{\currz}(1,a) \cdot\frac{p_1(y)}{w_1(0) p_0(y)+ w_1(1) p_1(y)} \right).
    \end{align}
It follows that under ${\bf C_0}$, $\deltamm$ selects hypothesis $0$ ($a=0$) iff
$
f\left( \frac{p_0(y)}{p_1(y)}\right) \geq \frac{L_{\currz}(1,0)}{L_{\currz}(0,1)} \text{\ where\ } 
f(r) = r \cdot \frac{(1-w_1) r + w_1}{(1-w_0) r + w_0} 
$
with $w_j = w_j(1)$. Straightforward calculus gives that, irrespective of the choice for $w_0$ and $w_1$,  $f(r)$ is strictly increasing in $r$, so that there is a unique point $r^*$ depending on $w_1$ and $w_0$, such that for all $r \geq r^*$, $\drmm{b}(y) = 0$ and for all $r < r^*$,   $\drmm{b}(y) = 1$. 
Further, for any pure Savage-Dickey e-collection with $w_1 = w_0$, we have $f(r) = r$, i.e.\ we select action $1$ if 
\begin{align}\label{eq:SDisBayesUniform}
    p_1(y) \cdot (L_{\currz}(1,0))^{\gamma} \geq  p_0(y) \cdot (L_{\currz}(0,1))^{\gamma} 
\end{align}
for $\gamma =1$. This is equivalent to taking the Bayes decision under a uniform prior. On the other hand, for the LR e-collection, $\drmm{b}(y)$ selects action $1$ if the above holds for $\gamma = 1/2$. 
Finally, if the loss  is symmetric, $L_{\currz}(0,1) = L_{\currz}(1,0)$, then in fact we select action $1$ iff $p_1(y) > p_0(y)$ even for general $w_1 \neq w_0$.

How about our basic risk assessment (\ref{eq:riskassessment}) for general decision rules, including these minimax ones?
While (\ref{eq:riskassessment}) does not give rise to easily interpretable formulae that work for all generalized Savage-Dickey e-variables, under {\bf C0}  we do get intuitive expressions for pure Savage-Dickey  e-variables with $w_0=w_1=1/2$ the uniform prior, and  for LR e-collections. 
For the former, we can replace the `$\max$' by a `$+$' using {\bf C0} as in (\ref{eq:twohypothesessd}) and we get:
\begin{align}
 &   L_{\currz}(0,\dr{b}(y)) \cdot\frac{p_0(y)}{\frac{1}{2} p_0(y)+ \frac{1}{2}  p_1(y)} +  L_{\currz}(1,\dr{b}(y)) \cdot\frac{ p_1(y)}{\frac{1}{2}  p_0(y)+ \frac{1}{2}  p_1(y)}  = \nonumber
  \\   &  2 \cdot
  L_{\currz}(0,\dr{b}(y)) \cdot w_{\textsc{u}}(0\mid y)  +  2 \cdot 
  L_{\currz}(1,\dr{b}(y)) \cdot w_{\textsc{u}}(1 \mid y) 
   \label{eq:twohypothesessdb}
    =   2\cdot  {\bf E}_{\theta \sim W_{\textsc{u}} \mid y} \left[  L_{\currz}(\theta,\dr{b}(y))
    \right] 
    \end{align}
    where $w_{\textsc{u}}(j \mid y) = \frac{p_j(y)}{p_0(y)+ p_1(y)}$ is the posterior based on uniform prior  $w_{\textsc{u}}$ with $w_{\textsc{u}}(0) =w_{\textsc{u}}(1)= 1/2$, showing that our bound is always within a factor $2$ of the best Bayes-with-uniform-prior bound and hence within a factor of less than 2 of the standard frequentist minimax bound for this setting, obtainable by a procedure that depends on the specific $L$ of interest.  
%
    For the LR e-variables, we similarly get, refining (\ref{eq:lrassessment}),
\begin{align}\label{eq:twohypothesessdc}
    &  L_{\currz}(0,\dr{b}(y)) \cdot\frac{p_0(y)}{p_1(y)} +  L_{\currz}(1,\dr{b}(y)) \cdot\frac{ p_1(y)}{p_0(y)}  \nonumber
  = \\  &  
    L_{\currz}(0,\dr{b}(y)) \cdot 
  \frac{w_{\textsc{u}}(0 \mid y) }{w_{\textsc{u}}(1 \mid y) }
  +  L_{\currz}(1,\dr{b}(y)) \cdot 
  \frac{w_{\textsc{u}}(1 \mid y)}{{w_{\textsc{u}}(0 \mid y) }}  
 = {\bf E}_{\theta \sim w_{\textsc{u}} \mid y} \left[ 
    \frac{1}{w_{\textsc{u}}(\dr{b}(y) \mid y )} \cdot L_{\currz}(\theta,\dr{b}(y))
    \right]. 
    \end{align}
We note that for $\delta :=\deltamm$ and the symmetric case $L_{\currz}(0,1) = L_{\currz}(1,0)$, we choose action 1 iff $p_1(y) \geq p_0(y)$ so that $w_{\textsc{u}}(\drmm{b}(y) \mid y)\geq 1/2$: the LR bound (\ref{eq:twohypothesessdc}) dominates the uniform-prior Savage-Dickey bound (\ref{eq:twohypothesessdb}), and $ 1/w_{\textsc{u}}(\drmm{b}(y) \mid y )$ in (\ref{eq:twohypothesessdc}) is a {\em luckiness factor}, bounded by $2$: if we happen to assign high posterior to the selected hypothesis, we are `lucky' and our loss assessment will get tighter; still, it is valid irrespective of the value of this posterior. If asymmetric losses are observed, $ 1/w_{\textsc{u}}(\drmm{b}(y) \mid y )$  can in general be both larger and smaller than 2, and the LR risk assessment bound does not dominate the Savage-Dickey one.
\end{example}
\subsection{Conditionalist-Frequentist E-Specifications --- not  E-Processes}
\label{sec:conditionalistfrequentist}
We first continue the
  {\em Berger-Brown-Wolpert\/} (BBW) simple-vs.-simple hypothesis testing setting of Example~\ref{ex:bbwfirst}, later extending it to more general conditionalist frequentist inference. Let $V := p_1(Y)/p_0(Y)$. We will assume  that (a)  $V$ is  strictly increasing in $Y$, which (b) has a (Lebesgue) density under $P_0$ and $P_1$. Both assumptions are merely for simplicity; for example (a) can be dropped by directly analyzing $V$ instead of $Y$,  i.e.\ `the likelihood ratio of the likelihood ratio'.

Fix some $w$ with $0 < w < 1$. Since, for any $y^*$, the sum of both integrals below is $1$, there must be a specific $y^*$  such that 
\begin{equation}\label{eq:ystar}
\int_{y < y^*} ((1-w) p_0(y) + w p_1(y)) d y = w \ ; \  \int_{y \geq  y^*} ((1-w) p_0(y) + w p_1(y)) dy = 1-w, 
\end{equation}
%
so that, upon setting
\begin{align}\label{eq:kbbw}
S^{\bbw(w)}_0(y) = {\bf 1}_{y \geq y^*} \frac{(1-w) p_0(y) + w p_1(y)}{(1-w) p_0(y)} 
\ \ ; \ \ 
S^{\bbw(w)}_1(y) = {\bf 1}_{y < y^*} \frac{(1-w) p_0(y) + w p_1(y)}{w p_1(y)}
\end{align}
we can verify that ${\bf E}_{Y \sim P_\theta} [S^{\bbw(w)}_{\theta}] = 1$ for $\theta \in \{0,1\}$ so that  $\{S^{\bbw(w)}_\theta: \theta \in \{0,1\} \}$ is an e-collection.
The corresponding e-posteriors only give useful  risk assessment bounds (\ref{eq:riskassessment}) in quite restricted situations, however. To be sure, the bound (\ref{eq:riskassessment}) holds (as it must by Proposition~\ref{prop:TheBound}), but it must be trivial for all decision rules except one as soon as we assume --- as will be the case for every Wald-Neyman-Pearson loss, i.e. with $L(0,1) > 0, L(1,0) > 0$. The reason is that 
$\bar{P}^{\bbw(w)}(\theta =1 \mid y) = \infty$ if $y \geq y^*$ and $\bar{P}^{\bbw(w)}(\theta = 0 \mid y) = \infty$ if $y < y^*$: for every loss function $L_{\currz}$ and decision rule $\dr{b}$ there will be $y$ such that $\bar{R}({y,\dr{b}(y)}) = \infty$, unless $L(0,0) \leq 0, L(1,1) \leq 0$  and $\dr{b}$ is the {\em $\bbw(w)$-decision rule}, which  satisfies $\delta^{\bbw(w)}(y) =1$ iff $y \geq y^*$. This is, in fact, the only decision rule that BBW themselves study. 
Note that it must then also be the case that
the minimax decision rule $\drmm{b}$ based on $\bar{P}^{\bbw(w)}$ must coincide with $\delta^{\bbw(w)}$, and thus does not depend on the loss $L_{\currz}$ of interest. This indicates that, for many loss functions, it will not be a very good rule --- it is simply the only rule for which the risk assessment (\ref{eq:riskassessment}) is  nontrivial (and then evaluates to (\ref{eq:kbbwassessment})).  To be sure, BBW provide a result implying that  if the loss function $L$ satisfies ${\bf C0}$, then we may set $w$ as a function of this loss function so that the rule $\delta^{\bbw(w)} = \deltamm$ is to some extent reasonable after all in the following sense: it coincides with the Bayes decision rule  $\deltabayes(y)$ based on prior $W$ with $w(1) = w$ as long as $y \not \in [y^-,y^+]$ for some specific values $y^ - \leq y^+$. They call $[y^-,y^+]$ the {\em no-decision region\/}
and notice that in practice it is often quite small.

This support for $\delta^{\bbw(w)}$ notwithstanding, we feel that the pure Savage-Dickey (with $w=1/2$) and LR e-posteriors should be preferred: at the small price of the slightly larger risk assessments (\ref{eq:twohypothesessdb}) and (\ref{eq:twohypothesessdc}), they (a)  allow us to use these assessments for general decision rules rather than the potentially defective (since independent of $b$) $\delta^{\bbw(w)}$, and  (b) they provide e-processes rather than just e-specifications (it is easy to see that the likelihood ratios (\ref{eq:eprocesses}) corresponding to BBW e-specifications have $q^{[n]}_{\theta}$ depending on $n$ --- BBW e-variables do not remain e-variables under optional stopping, generalized Savage-Dickey e-variables do).

\paragraph{BBW's Original, Conditionalist-Frequentist Interpretation}
For general $p_0$ and $p_1$ 
as above, we may partition the sample space $\cY$ into regions $\{\cY_\alpha: \alpha \in [0,1]\}$ where $\cY_{\alpha} = \{y^+_{\alpha},y^{-}_{\alpha} \}$ are  chosen such that $P_0(Y \geq y^+_{\alpha}) = \alpha$ and $P_1(Y \leq y^-_{\alpha}) = \beta(\alpha)$ for some function $\beta:[0,1] \rightarrow [0,1]$. We may think of observing $Y=y$ as first observing $\alpha$ so that $y \in \cY_{\alpha}$, and only then observing $Y=y$ itself.
We may then, by construction, after observing 
$\alpha$, perform {\em a conditional frequentist test\/} with Type-I and Type-II `conditional'  error probabilities $\alpha$ and $\beta(\alpha)$, by rejecting $\theta =0$ if $Y = y^+_{\alpha}$ and accepting  if $Y = y^-_{\alpha}$. The BBW approach was originally conceived as a specific way to construct such a partition, with the special property that for any $\alpha$ we have that $y^+_{\alpha} \geq y^*, y^-_{\alpha} \leq y^*$ (so that $\deltamm$ invariably rejects if $y^+_{\alpha}$ is observed for some $\alpha$) and
$$
\alpha = \frac{(1-w) p_0 (y^+_{\alpha})}{(1-w) p_0(y^+_{\alpha}) + w p_1(y^+_{\alpha})}, \ \ 
\beta(\alpha) = \frac{w p_1 (y^-_{\alpha})}{(1-w) p_0(y^-_{\alpha}) + w p_1(y^-_{\alpha})}. 
$$
These are exactly our BBW e-posteriors $\bar{P}^{\bbw(w)}(0 \mid y^+_{\alpha})$ and
$\bar{P}^{\bbw(w)}(1 \mid y^-_{\alpha})$ --- BBW in fact call them conditional error probabilities, because that is how they derived them, and note that they are formally identical to Bayesian posteriors; the interpretation in terms of e-posteriors is ours.
\paragraph{Kiefer and Vovk Confidence Estimators}
Essentially the same analysis as for BBW conditional error probabilities also applies to general {\em conditionalist frequentist e-variables}, based on Kiefer's and Vovk's conditional confidence estimators, to which we now turn: every such confidence estimator gives a valid e-specification, hence e-posteriors for each $n$, but in general these will not be e-processes, and they can evaluate to infinity, hence only giving  useful, nontrivial risk assessments for a restricted set of loss functions and decision rules --- although this set may not be as restricted as in the BBW case, where only a single rule,  $\delta^{\bbw(w)}$ gave  nontrivial risk assessments.  

Consider a partition $\{\cY_{\alpha}: \alpha \in [0,1]\}$ of $\cY$
together with a function $\textsc{CI}(y)$  that maps each  $y \in \cY$ into a subset  of $\Theta$.
In case $\Theta$ is a continuum, we usually take $\textsc{CI}(y)$ to be an interval, so that $\textsc{CI}(y) = [\theta_L,\theta_R]$ with endpoints $\theta_L$ and $\theta_R$ determined by $y$. We say that $(\{\cY_{\alpha}: \alpha \in [0,1]\},\textsc{CI})$ define a {\em  Kiefer-type conditional frequentist confidence set estimator\/} 
if for all $\theta \in \Theta$, all $\alpha \in [0,1]$,
$$
P_{\theta}(\theta \not \in \textsc{CI}(Y) \mid Y \in \cY_{\alpha}) \leq \alpha.
$$
Such confidence estimators were  first considered by 
\cite{kiefer1976admissibility}.  If $\beta(\alpha) = \alpha$, the BBW approach can be reinterpreted as a Kiefer-type confidence estimator: we set $\cY_{\alpha}$ as above and
\begin{equation}\label{eq:kiefermans}
\textsc{CI}(y) = \{ \delta^{\bbw(w)}(y)\} \text{\  so that\ } \textsc{CI}(y^+_{\alpha}) = \{1\}, \textsc{CI}(y^-_{\alpha}) = \{0\}.\end{equation} 
If $\beta(\alpha) \neq \alpha$ then the BBW approach does not give a Kiefer-type confidence estimator.
However, \cite{Vovk93} provided a notion of confidence estimator that generalizes both BBW's conditional error probabilities and Kiefer-type confidence estimators: for any function $\textsc{CI}(y)$ that maps each $y \in \cY$ into a subset  of $\Theta$ combined with any function $\hat\alpha(Y)$ taking values in  $[0,1]$, we say that $\textsc{CI}$ and $\hat\alpha$ define a {\em Vovk-type confidence set estimator\/} if for all $\theta$, 
\begin{equation}\label{eq:vovkconfidence}
{\bf E}_{Y \sim P_{\theta}} \left[ {\bf 1}_{\theta \not \in \textsc{CI}(Y)} \cdot \frac{1}{\hat\alpha(Y)}\right] \leq 1.
\end{equation}
\cite{Vovk93} defines confidence set estimators in this partition-free manner, motivated by a game-theoretic interpretation. He gives a (straightforward) proof showing that any confidence set estimator in the sense of Kiefer is also a confidence set estimator in the sense of Vovk, with $\hat\alpha(y)$ set to the $\alpha$ such that $y \in \cY_{\alpha}$. Similarly, if we set  $\textsc{CI}$ as in (\ref{eq:kiefermans}) and, for each $\alpha \in [0,1]$, $\hat\alpha(y^+_{\alpha}) := \alpha, \hat\alpha(y^-_{\alpha}) = \beta(\alpha)$, then it is easily seen that (\ref{eq:vovkconfidence}) hold with equality for both $\theta \in \{0,1\}$, i.e.\ we obtain a Vovk-type confidence estimator. 

Now, if $(\textsc{CI}(Y),\hat\alpha(Y))$ is a confidence set estimator in sense of Vovk, it directly follows from (\ref{eq:vovkconfidence}) that for all $\theta$, $S_{\theta} = {\bf 1}_{\theta \not \in \textsc{CI}(Y)} 1/\hat\alpha(Y)$ is an e-variable, so $\bar{P}(\theta \mid y)$ is an e-posterior. 
Just as in the specific BBW scenario, such a posterior may become $\infty$.

\section{Condition Zero and Capped Posteriors}
\label{sec:capped}
Consider Example~\ref{ex:normal} again, the normal location family  with the squared error loss (\ref{eq:squarederror}) as in Example~\ref{ex:gaussexample}. If we base our risk assessment directly on the Savage-Dickey e-posterior as defined there, we get bounds that, while correct, are not always sharp. This is due to the fact that, as seen from Figure~\ref{fig:eposterior},  $\bar{P}(\theta \mid Y) = S^{-1}_{\theta}(Y)$ can become very high for $\theta$ close to $\hat\theta$ --- a milder form of the issues with the confidence-estimator based e-variables of the previous section, for which it can become even infinite. There, the issues were avoided if the loss function satisfied Condition Zero. It turns out that a weakened version of this condition is sufficient to avoid this problem, for decision problems more general than testing. We now introduce a new concept, and weaken and  generalize Condition Zero with this in mind: 
\begin{definition}{\bf [Capped E-Posterior; Condition Zero]}
\label{def:capped}
For any given e-posterior $\bar{P}(\theta \mid y)$, its {\em capped\/} version $\capped{P}(\theta \mid y)$ is defined by setting, for all $\theta \in \Theta$, $y \in \cY$: $\capped{P}(\theta \mid y) := \min \{1, \bar{P}(\theta \mid y)\}$.
Let $\delta$ be a decision rule. We say that $\dr{b}$ satisfies {\bf C0}, i.e.\ Condition Zero,  if 
\begin{equation}\label{eq:capped}
\sup_{\theta \in \Theta} \bar{P}(\theta \mid y) L_{\currz}(\theta,\dr{b}(y)) =
\sup_{\theta \in \Theta} \capped{P}(\theta \mid y) L_{\currz}(\theta,\dr{b}(y))
\end{equation}
for every $y \in \cY$.
We say that the {\em loss function\/} $L_{\currz}$ itself  satisfies Condition Zero if there exists a decision rule $\delta$  such that (\ref{eq:capped}) holds. 
\end{definition}
Our original Condition Zero in simple-vs.-simple testing was defined as a condition on a given loss function $L_{\currz}$; the present condition is strictly weaker as a condition on loss functions and generalizes the idea to decision rules $\dr{b}$. Why would we want to impose such a condition?

First note that setting $S_{\theta} = 1$ for all $\theta \in \Theta$ gives a trivial, yet valid, e-collection; it provides no indication at all about $\theta$ based on data $y$. This suggests that, when using a nontrivial collection, based on data $y$, we cannot make any inferences about  $\theta$ for which $S_{\theta}(y) < 1$ (i.e.\ $\bar{P}(\theta \mid y) > 1$): our e-variable provides less evidence against $\theta$ than what can be obtained trivially: in a sense, $y$ gives us no information at all for making inferences about $\theta$. 
It is therefore of interest to restrict to decision rules satisfying the new Condition Zero. For such problems, our risk assessment will {\em never \/} be based on evaluating $\bar{P}(\theta \mid y)$ at a value of $\theta$ for which $\bar{P}(\theta \mid y) > 1$, i.e.\ for which it really does not give any information about this $\theta$. 
Also, it may be argued that the capped e-posterior is a more intuitive tool than the e-posterior: 
for any $\alpha \geq 0$, the set $\{\theta \in \Theta: \bar{P}(\theta \mid y) \geq \alpha\}$ is a $1-\alpha$-confidence interval \citep{Grunwald22} (see  Figure~\ref{fig:eposterior}, middle panel), whereas for sets containing $\theta$ with $\bar{P}(\theta \mid y) \geq 1$, no clear interpretation seems to exist. Thus, by insisting on {\bf C0} one really insists on inferences that can be based on the more intuitive $\capped{P}$ while still satisfying (\ref{eq:riskassessment}).

\begin{example}{\bf [Two-Point prior for  exponential families: Condition Zero holds]} \label{ex:normalzero} The generalized Savage-Dickey two-point prior-based e-posterior that we employed for the Gaussian location family in Example~\ref{ex:normalcontinued} gives rise to e-collections such that the corresponding minimax  $\drmm{b}(y) =\hat\theta(y)$, also satisfy Condition Zero whenever the $b$-weighted squared error (\ref{eq:squarederror}) is used: independently of $n$, $\bar{P}(\theta \mid y) \approx 0.68$ for the maximizing $\theta$, as was established below (\ref{eq:superbound}). For exponential families, we show in the Supplementary Material (Section~\ref{app:normalproofs}, inside the proof of Theorem~\ref{thm:exp}) that the same holds if the model $\{P_{\theta} : \theta \in \Theta^{\circ}\}$ is restricted to an (arbitrary) compact subset of the parameter space and $n$ is large enough. We do not know if the result can be extended beyond exponential families, since its proof makes heavy use of their properties. 
\end{example}
In practice, the DM may be presented with a loss function for which Condition Zero does not hold, for any reasonable decision rule, and the high-risk problem looms.
What to do in such a case? 
There are two alternative solutions. As to the first (we describe the second underneath Example~\ref{ex:logn}), one can (sometimes vastly) improve risk assessments by slightly modifying any given e-variable $S_{\theta}$ to another e-variable $S'_{\theta}$ that provides almost as much evidence against $\theta$ as $S_{\theta}$ whenever $S_{\theta}$ provides more evidence against $y$  than the trivial $1$ (i.e.\ when $S_{\theta}(y) > 1$), and almost as much as the trivial $S=1$ otherwise.  
The corresponding $\bar{P}'(\theta \mid y)= S'^{-1}_{\theta}(y)$
will then be similar to $\bar{P}(\theta \mid y)$ if $\bar{P}(\theta \mid y) \leq 1$ and bounded by a small constant otherwise. Three straightforward ways of doing this are provided by setting  $S'_{\theta}$ equal to any of the following options: 
\begin{align}
 &   S^{*}_{\theta}(y) := \frac{\max \{S_{\theta}(y), 1 \}}{ 
    {\bf E}_{Y \sim P_{\theta}} [\max \{S_{\theta}(Y), 1 \} ]
    } \ \ , \ \ S^{\circ}_{\theta}(y) := \frac{1}{2} \cdot  \max \left\{ 1,  S_{\theta}(y) 
    \right\} \nonumber \\ \label{eq:dampen}
    &  S^{[\gamma]}_{\theta}(y) := (1-\gamma) + \gamma S_{\theta}(y) \text{\ for \ } \gamma = 1/2.
\end{align}
One immediately verifies that all three options produce e-variables: $S^{[\gamma]}$ produces one for all $\gamma$, because mixtures of e-variables are e-variables. Using that the expectation in $S^*_{\theta}$ is bounded by $2$ (to see this, upper bound $\max$ by a sum), we see that  $S^{\circ}$ is dominated by both $S^{*}_{\theta}$ and $S^{[1/2]}_{\theta}$ hence must also be an e-variable and we also see that for all three choices, we have  $\sup_{\theta} \bar{P}'
 (\theta \mid y) \leq 2 \sup (S^{\circ}_{\theta}(y))^{-1} = 2$. 
$S^*_{\theta}$ may be the most direct solution, but it can turn an e-process into an e-specification that is not an e-process anymore, whereas $S^{[1/2]}_{\theta}$ preserves the e-process property. Since $S^{[1/2]}_{\theta}$ dominates $S^{\circ}_{\theta}$, we prefer it; it is depicted in Figure~\ref{fig:eposterior}.
%

\begin{example}{\bf [Improving Pure Savage-Dickey  for  Exponential Families]}\label{ex:logn}
We illustrate 
the idea above for the  case that $\bar{P}(\theta \mid y) = 1/S_{\theta}(y)$ is the Savage-Dickey e-posterior based on an arbitrary prior $W$ on the Gaussian location family.
%
Let $\bar{P}^{[\gamma]}(\theta \mid y) := 1/S^{[\gamma]}_{\theta}(y)$ be the corresponding `dampened' e-posterior as defined in (\ref{eq:dampen}). 
In the Supplementary Material, Section~\ref{app:logn} we describe this setting in more detail and extend it to 1-dimensional exponential families, also showing (in Proposition~\ref{prop:KLahoy}) that for any $0 <  \gamma \leq 1$ the MLE $\hat\theta$ is the $\bar{P}^{[\gamma]}$-e-posterior minimax estimator irrespective of the weight $b$ in the loss (\ref{eq:squarederror}).
Now for simplicity let  $W$ be  a Gaussian prior with some fixed precision $\lambda$ as in Example~\ref{ex:normal} and with prior mean $\theta_0 = 0$. We also show in the Supplementary Material (Proposition~\ref{prop:dampening}), that, for the choice $\gamma = 1/2$ that we advocated above, the bound (\ref{eq:TheBoundA}) and (\ref{eq:TheBoundB}) hold with 
\begin{equation}\label{eq:boundy}
\bar{R}(\deltamm) = \bar{R}(Y,\hat\theta(Y)) = \frac{2}{n} \left(\frac{1}{2} \log 
\frac{n + \lambda}{\lambda} + 
\frac{n \lambda}{n+\lambda} \hat\theta(Y)^2 
\right),
\end{equation}
losing just a ${\log n}$ factor compared to the e-posterior that achieves the standard frequentist minimax bound up to a constant factor,  and that we considered in Example~\ref{ex:normalcontinued}. Again we extend this result to general 1-dimensional families with smooth and strictly positive prior densities (and again we do not know if it extends to more general 1-dimensional families).  
In contrast, for the original Savage-Dickey e-posterior $\bar{P}(\theta \mid y)$ we get a much larger  $\bar{R}(\deltamm) \succeq  n^{-1/2}$, as can be seen by evaluating $  \cdot \bar{P}(\theta \mid y) (\theta - \hat\theta)^2$ at $\theta = \hat\theta + n^{-1/2}$.  By `dampening' the posterior we improved the factor of order  $ n^{-1/2}$ to a factor of order $(\log n)  n^{-1}$. 
\end{example}
As a second solution to the high-risk problem, we can alternatively {\em modify\/} our risk assessment (\ref{eq:riskassessment}) 
for decision problems and rules that do not satisfy {\bf C0}.
We then find that if we replace the e-posterior in the risk assessment $\bar{R}(\delta) $ as in (\ref{eq:riskassessment}) by its capped version, the risk assessment itself may become invalid (i.e.\ $\bar{R}(\delta)$ does not satisfy (\ref{eq:TheBoundA})/(\ref{eq:TheBoundB})), but, for arbitrary data-dependent decision problems, $2 \bar{R}(\delta)$ does remain valid: 
\begin{proposition}\label{prop:TheBoundb}
Set, for arbitrary multi-loss decision problem and  rule $\delta$, not necessarily satisfying ${\bf C0}$, 
\begin{align}\label{eq:cappedriskassessment}
     \capped{R}(Y,A) := \sup_{\theta \in \Theta} \capped{P}(\theta \mid y)  \cdot L_{\currz} (\theta,a)\ \ , \ \ \capped{R} (\delta) := \capped {R}({Y,\dr{B}(Y)}).
\end{align} 
Then (\ref{eq:TheBoundA}) and (\ref{eq:TheBoundB}) hold with $2 \capped{R} (\delta)$ in the role of $\bar{R}(\delta)$.
\end{proposition}
\begin{proof}
Calculating $\capped{R}({y,a})$ based on  the capped posterior 
is equivalent to calculating $\bar{R}({y,a})$ based on the modified posterior $\bar{P}'(\theta \mid y) = (S^{\circ}_{\theta}(y))^{-1}$ as defined in (\ref{eq:dampen}) and multiplying the answer by $2$. Since $\bar{P}'(\theta \mid y)$ is a standard e-posterior and no requirements on $\delta$ are made, we can use Proposition~\ref{prop:TheBound} based on $\bar{P}'(\theta \mid y)$ to get the required result. 
\end{proof}
This solution to the high-risk problem of some e-collections shows that, for some collections, the original e-posterior in its capped version may not be so bad after all. This solution may in fact be preferable, since it does not interfere with the use of e-values in optional continuation: we can keep multiplying e-values as explained in the next section, and get a factor of 2 in the bound for the final product. In contrast, if we use any of the modifications in (\ref{eq:dampen}), we must modify the e-variable at each study, making the final (i.e. when we stop adding studies) product of e-values potentially much smaller\vspace*{-1 mm}. 

\section{Additional Background and Discussion}
\label{sec:discussion}
\paragraph{The General Definition of E-Values}
The definition of e-collection, e-posterior, risk assessment (\ref{eq:riskassessment}) and e-posterior minimax decision rule readily generalize to models with nuisance parameters, composite hypothesis testing and nonparametric settings. All are covered by the following extended definition:
let $\cP$ be a set of distributions for random variable $Y$ and consider a parameter (vector) of interest defined by a function $\delta: \cP \rightarrow \Theta$ with $\Theta \subset \reals^k$ for some $k> 0$.  We now call $\{ S_{\theta}: \theta \in \Theta \}$ an {\em e-collection\/} for $\cP$ if for all $\theta \in \Theta$, $S_{\theta}= S_{\theta}(Y)$ is a nonnegative function of $Y$ and for all $P \in \cP$ with $\delta(P) = \theta$, we have the corresponding analogue of (\ref{eq:basic}), i.e.\ for {\em all\/} such $P$:
\begin{equation}
    \label{eq:evalgen}
{\bf E}_{P}[S_{\theta}] \leq 1. 
\end{equation}
The e-posterior $\bar{P}(\theta \mid y)$ relative to such a collection is still defined as $S^{-1}_{\theta}(y)$, and the definition of risk assessment (\ref{eq:riskassessment}) and e-posterior minimax rule (\ref{eq:minimax}) remains unchanged. It is easily checked that the results that validate the e-posterior risk assessment and minimax rule, Proposition~\ref{prop:TheBound} and Proposition~\ref{prop:TheBoundb}, still hold in this more general setting\vspace*{-1 mm}. 

\paragraph{The Original Interpretation of E-Variables}
E-variables, while being implicit in earlier work going back at least to \cite{darling1967confidence}, were largely unknown until 2019, when the concept was given a name and was popularized and extensively developed by papers such as \citep{GrunwaldHK19,wasserman2020universal,Shafer:2021,VovkW21} 
(the first versions of all four papers came out in 2019, and, together with \citep{howard2020time,howard2021time}, which first came out in 2018 and in which e-values played a central role without an explicit name, may arguably be viewed as {\em the\/} pioneering papers of the field). 
In the mean time they have been studied in a wide variety of contexts \citep{henzi2021valid,RenB22,BatesJSS22} and even had an international workshop devoted to them.  Originally, they were used as tools to extend traditional Neyman-Pearson tests to situations with {\em optional continuation\/} while keeping Type-I error guarantees: suppose one observes a sequence of studies concerning the same null hypothesis $\{P \in \cP: \delta(P) = \theta_0\}$. One can then multiply the e-values derived from the individual studies. Broadly speaking, the product will be a new e-variable, irrespective of whether the decision to perform a subsequent study depends on previous study outcomes (the `continuation rule'), and irrespective of the rule used to decide when to stop performing new studies. 
Since, for any fixed $\alpha$, the probability, under the null, that an e-variable gets larger than $1/\alpha$, is bounded by $\alpha$, one can use e-values, combined by multiplication, as an alternative to p-values in a setting in which optional continuation of experiments is allowed (note that this is always possible, whereas optional stopping {\em within\/} a study is only possible if we deal with e-processes rather than e-variables, as in Section~\ref{sec:taxonomy}).  

\paragraph{Composite Testing}
In light of this original motivation, much initial work on e-values focused on hypothesis testing in Neyman-Pearson (NP) style, but under optional continuation. Just as in the original NP paradigm, Type-I error of decision rules was required to be bounded by a pre-given significance level $\alpha$, and other desiderata such as small Type-II error only came second. This is different from the treatment of simple-vs.-simple testing in the present paper. Here, as in Bayesian and minimax approaches, we do allow asymmetric losses ($L_{\currz}(0,1) \neq L_{\currz}(1,0)$), but these potentially different losses appear in the risk assessment (\ref{eq:riskassessment}) in a symmetric fashion after all. 
 
Existing papers on e-values within the original paradigm in fact emphasize composite hypotheses: 
the theoretical papers \cite{GrunwaldHK19,perez2022estatistics} and the practical \citep{TurnerLG21,TurnerG22} consider situations in which the null has nuisance parameters (such as, for example, the variance in the t-test or the proportion in $2 \times 2$ tables). For such cases (the null $\cP= \{P_{\lambda}: \lambda \in \Lambda \}$ is characterized by a (potentially) high-dimensional parameter vector $\lambda$ ; the parameter of interest $\delta$ is of (potentially) much smaller dimension, often a scalar), \cite{GrunwaldHK19} offer a generic method to construct e-variables as in (\ref{eq:evalgen}) for such situations via the {\em reverse information projection}. This method is employed by all papers mentioned above to construct   e-variables, with the explicit goal of providing Type-I error bounds under optional continuation in the Neyman-Pearson style;  but we can equally use them to provide e-posterior risk assessments and minimax rules with general loss functions as in this paper --- in our examples above, we only used `simple' e-collections for ease of illustration.
Of particular interest is the paper \citep{Grunwald22}, which may be viewed as a companion to the present one. It also extends the use of e-variables to deal with post-hoc specified loss functions 
but it sticks to a variation of the Neyman-Pearson tenet: there is a distinction between Type-I risks (expected losses), on which a hard constraint is imposed (similar to the significance level $\alpha$ as a constraint on Type-I error {\em probability\/} in traditional testing), and Type-II risks, which are minimized subject to this constraint\vspace*{-3 mm}.
\paragraph{Nonparametrics}
The same generalization (\ref{eq:evalgen}) covers  nonparametric settings. For example, 
in Cox's celebrated proportional hazards model, the $P \in \cP$ specify an underlying continuous time process with dynamics conditional on some covariate vector $\vec{x}$. While the full likelihood of this process is intractable, \cite{TerschurePLG21} show that, for any vector $\theta$ of covariate coefficients, the ratio between any two partial likelihoods defined relative to the same set of covariates constitute an  e-process. This implies that all the results on e-posterior risk assessment of the present paper also pertain to ratios of Cox's partial likelihoods. 
In fact, for any $\cP$ and any $\theta \in \Theta$, any ratio between conditional, marginal or partial likelihoods that has the same distribution under all distributions in $\cP$ with $\delta(P) = \theta$ is trivially an e-variable, and can be used to define an e-collection.  
We regard all this as an additional advantage of e-posteriors over Bayesian ones: strictly speaking, a Bayesian needs to specify the {\em full\/}  likelihood for each distribution in $\cP$ in order to get a valid a posterior --- as has been argued by e.g. \cite{Larry2012}, this may be viewed as `overkill' and  poses practical (and unnecessary) difficulties for Bayesian nonparametric inferences especially when the samples are small. In contrast, \cite{waudby2020estimating} provides an e-collection which can be used to learn the mean of a sequence of bounded random variables based on a small sample without {\em any\/} assumptions on the underlying random variables except boundedness and existence of a common mean. 

\paragraph{Conclusion and Future Work}
We have presented the idea of representing uncertainty by e-posteriors based on e-collections. While e-collections already existed, the idea to use them in risk assessments such as (\ref{eq:riskassessment}) and to define e-posterior minimax rules as in (\ref{eq:minimax}) is, to our knowledge, completely new.
There are of course many open questions regarding this quasi-conditional methodology. 
The main one is perhaps: {\em what e-collection should be used in practice?\/}
If one takes as one's goal ``collecting as much evidence as possible in subsequent studies against a single $\theta \in \Theta$'', and one is able to formulate a prior on $\Theta$ representing a `best guess' (note that $\Theta$ does not necessarily index a set of distributions; rather a set of properties) then there is a good case to be made for the GRO (growth-rate optimal) e-variable as defined by \cite{GrunwaldHK19}, which is always well-defined; we refer to that paper for extensive discussion. For simple nulls, the GRO e-variable is always of pure Savage-Dickey form. 
However, if we are interested not in testing but in estimation and want an e-posterior over the full range of $\theta \in \Theta$, then the e-posterior based on the GRO e-variables may not be intuitively optimal, as we saw in this paper: the pure Savage-Dickey e-variable based on a normal prior $W$ in the Gaussian location family (which is GRO relative to $W$) gives worse risk assessments than the one based on the two-point prior as in Example~\ref{ex:discrete}. We do not have a general recommendation (in the style of `reference priors' for Bayesians) for e-collections, and obviously more research is needed here. This paper is really just  a first, exploratory one, comparing various possibilities such as two-point, capped- and dampened e-collections. A second related important question is: 
in what situations is the bound (\ref{eq:TheBoundB}) sufficiently tight to be practically relevant? Essentially, if we are in the setting that (bookie, or policy makers) provide $B$ such that $B \cdot \bar{R}(\delta)$ is 
constant or varies only very little over $y$,  then the bound (\ref{eq:TheBoundA}) looks like a standard expected risk bound. If $\bar{R}(\delta)$ can vary highly depending on the data, then the bounds (\ref{eq:TheBoundA}) and (\ref{eq:TheBoundB}) are still valid, but the supremum in (\ref{eq:TheBoundB}) may make them rather weak. Can we get better bounds in such cases? And what if we let $B$ also vary with $\theta$ and $a$? Then, upon learning $\{B_{\theta,a}: \theta \in \Theta, a \in \cA \}$, DM might want to re-consider the originally planned decision rule $\delta$, significantly complicating the analysis. Related to this is the question what loss functions $L$ correspond most tightly to real-world problems at all. In our hypothesis testing examples, we considered asymmetric loss functions which seems reasonable, but in our estimation example, we used the squared error loss --- which is mathematically convenient and incentives one to output the posterior mean; but in practical situations one may be  more interested in very different, hard-to-formalize, problem-dependent losses. All in all, there is plenty of room for future work here! 
\bibliography{master,peter}
\pagebreak
\appendix
\section*{Supplementary Material}
\label{app:proofs}
In this Supplementary Material, in Section~\ref{app:normalproofs} we provide details and proofs concerning the two-point prior-based generalized Savage-Dickey e-variables for the Gaussian location family and its extension to one-dimensional exponential families. In the final Section~\ref{app:logn} we provide details and proofs concerning the 'pure' Savage-Dickey e-variables for Gaussian location and one-dimensional exponential families. In between in Section~\ref{app:mleminimax} we show how the MLE is e-posterior minimax both for Gaussian location with two-point prior and for general exponential families with pure Savage-Dickey.
\section{Two-point prior-based E-Collections}\label{app:normalproofs}
\subsection{A Theorem for One Dimensional Exponential Families}
Here we present Theorem~\ref{thm:exp}, which extends the result (\ref{eq:superbound}) for two-point prior-based e-collections of Example~\ref{ex:normalcontinued} to 1-dimensional exponential families. 
More precisely, we let $\{P_{\mean}: \mean \in \Theta\}$ be any 
regular \citep{BarndorffNielsen78} 1-dimensional exponential family 
given in any diffeomorphic parameterization (for concreteness, take the mean-value one)  and extended to $n$ outcomes by independence. We write the KL divergence between two members of the family defined on a single outcome as  $D(\mean'\| \mean)$.  We denote by $\hat\theta(y)$ the MLE based on data $y = x^{n}$, which is known to be  unique and equal to the empirical average $n^{-1} \sum_{i=1}^n \phi(X_i)$, with $\phi$ the sufficient statistic, whenever this average lies in $\Theta$, which for regular families is an open set.
Suppose we observe $Y= x^n$ with $\hat\theta(Y) \in \Theta$. 
We now use as our loss function 
$$
L_{\currz}(\theta,\breve\theta) := 2 \cdot  D(\breve\theta \| \theta),$$ 
which is parameterization independent; but with the mean-value parameterization, in the Gaussian location family, $2 D(\breve\theta \| \theta) = (\theta - \breve\theta)^2$ becomes the squared error loss as in Example~\ref{ex:gaussexample}; of course the factor $2$ is merely for mathematical convenience in the proofs. 

As we show below (\ref{eq:expfamproperty}), for every $C> 0$, every $\theta \in \Theta$,
there exist $\meanmin < \mean < \meanplus$  such that, consistent with the Gaussian case,
\begin{equation}
    \label{eq:postval}
D(\mean\| \meanplus) = \frac{C}{n^*}\ , \  
 D(\mean \| \meanmin) =
\frac{C}{n^*}. 
\end{equation}
As in the Gaussian example, we take $C=1$, and we construct, for each $n$, 
the e-variables $S^{[n]}_{\mean}(y) =  \frac{(1/2) p_{\meanmin}(y)+ (1/2) p_{\meanplus}(y)}{p_{\mean}(y)}$, obtaining again a generalized Savage-Dickey e-process. 
Based on the standard Taylor expansion 
\begin{equation}
    \label{eq:taylor} D(\theta^* \| \theta)
= \frac{1}{2} I(\theta^{\circ}) (\theta^* - \theta)^2 
\end{equation}
where $I(\cdot)$ is the Fisher information, which must hold exactly for some $\theta^{\circ}$ in between $\theta^*$ and $\theta$, we may expect risk assessments based on this e-variable to behave similarly as in the Gaussian case. Indeed, below we first prove the result (\ref{eq:superbound}) of the main text for the Gaussian case, after which we proceed to prove, heavily using (\ref{eq:taylor}), the following extension to regular 1-dimensional exponential families: 
\begin{theorem}\label{thm:exp}
Let $\Theta^{\circ}$ be an arbitrary compact subinterval of $\Theta$. Consider restricted model $\{P_{\theta} : \theta \in \Theta^{\circ} \}$ and the collection of e-processes $\{S_{\theta}: \theta \in \Theta^{\circ} \}$.
Then, for every $0 < c \leq 1$, there is a constant $C^*$ such that whenever $\hat\theta \in \Theta$, and $c \leq n^*/n \leq 1/c$, (\ref{eq:superbound}) still holds up to a vanishing $1 \pm O(n^{-1/2})$ factor. That is, there is $C^*$ such that, for all $n$, uniformly for all $x^n$ with $\hat\theta(x^n) \in \Theta^{\circ}$,
\begin{equation}
    \label{eq:superboundb}
    \max_{\theta \in \Theta^{\circ}} L_{\currz}(\theta,\hat\theta(x^n)) \cdot \bar{P}(\theta \mid x^n) = \frac{1}{n} \cdot \left(
    \frac{n^*}{n} \right) \cdot \exp^{n/n^*} \cdot f_n \cdot 0.53\ldots 
    \overset{\text{\rm if \ } n^*=n\ }{=} \frac{1}{n} \cdot f_n \cdot 1.45\ldots
\end{equation}
where $1 - \sqrt{C^*/n} \leq f_n \leq  1 +\sqrt{C^*/n}$.
\end{theorem}
Thus, even though now $\hat\theta$ may not be precisely equal to the e-posterior minimax estimator $\deltamm$ any more, we may use it as our decision rule anyway, and then still get, for each compact sub-model  $\Theta^{\circ}$, risk bounds that, at least for sufficiently large $n$, are of the same order as standard minimax frequentist risk bounds. 
\subsection{Proofs for Example~\ref{ex:normalcontinued},~\ref{ex:normalzero} in main text, and Theorem~\ref{thm:exp} above}
Below we prove all results that were stated in  the main text for the Gaussian location family and  1-dimensional  exponential families. 
We will freely, without further reference, use standard terminology and results concerning exponential families (all to be found in \citep{BarndorffNielsen78}), but we first do highlight two slightly nonstandard results that we will use repeatedly. The first (easily proved using steepness of regular exponential families) is that for each fixed $\mean' \in \Theta$, $D(\theta' \| \theta)$ is a continuous function of $\mean \in \Theta$ satisfying
\begin{equation}
\label{eq:expfamproperty}
\sup_{\mean< \mean'} D({\mean'} \| {\mean})= \sup_{\mean> \mean'} D({\mean'} \| {\mean}) = \infty.
\end{equation}
The second result we need is the {\em KL robustness property\/} \citep{Grunwald07} that holds for all regular exponential families: for any fixed $y = x^n$  such that $\hat\theta(y)$ is well-defined,  any $\theta \in \Theta$ and any prior $W$ on $\Theta$ ($W$ does not need to have a density), we have: 
\begin{equation}\label{eq:robustness}
\frac{p_{\theta}(y)}{p_W(y)}
= \exp(-n D(\hat\theta \| \theta) + D(P_{\hat\theta}^{(n)} \| P_W^{(n)})),
\end{equation}
where $P_W$ is the Bayes mixture distribution as defined in Example~\ref{ex:simplesd}, and, for any distribution $P$ underlying random process $X_1,X_2, \ldots$,  we denote the marginal distribution of the first $n$ outcomes as $P^{(n)}$.
\paragraph{Proof of (\ref{eq:superbound}) in Example~\ref{ex:normalcontinued} and Claim in Example~\ref{ex:normalzero}} 
Fix $C> 0$. By (\ref{eq:expfamproperty}) above, for general regular 1-dimensional exponential families,
there exist $\meanmin < \mean < \meanplus$  such that (\ref{eq:postval}) holds.

Now define, for $i = 1,2, \ldots$, $x^i \in \reals^i$, the e-processes $\Sl_{\mean}(x^i) = \frac{p_{\meanmin}(x^i)}{p_{\mean}(x^i)}$ and $\Sr_{\mean}(x^i)= \frac{p_{\meanplus}(x^i)}{p_{\mean}(x^i)}$
and $S_{\mean} := (1/2) \Sl_{\mean} + (1/2) \Sr_{\mean}$. 
While these e-processes provide valid e-variables at each sample size $i$ we have that when evaluated at sample size $n$ (on which their definition depends),
$\Sl_{\mean}$ 
coincides with the  {\em uniformly most powerful Bayes factor\/} for a 1-sided test at sample size $n$ and level $\alpha^*/2$, for any $n$ and $\alpha^*$ such that $C^{-}= - \log (\alpha^*/2)$, for $\cH_0 = \{P_{\mean}\}$ vs.\ $\cH_1 = \{P_{\mean'}: \mean > \mean\}$; similarly for $\Slr_{\mean}$ with $C^+$ and $\cH_1 = \{P_{\mean'}: \mean' < \mean\}$ respectively  \citep{johnson2013uniformly}. 

We first focus on the Gaussian location family for which  $D(\mean'\| \mean ) = (1/2) (\mean'- \mean)^2$ . 
$\bar{P}(\theta \mid y)$   can be written by  (\ref{eq:robustness})  with $W$ a prior putting mass $1/2$ on $\meanplus$ and $1/2$ on $\meanmin$. For the Gaussian location family, (\ref{eq:robustness}) then gives:  
\begin{align}\label{eq:gaussrobustness}
\bar{P}(\theta \mid y)  & = 
 \frac{p_{\mean}(y)}{\frac{1}{2}p_{\meanmin(y)} + \frac{1}{2}p_{\meanplus(y)} }  = 
 \frac{
2 e^{-n D(\hat\mean \| \mean)}}{e^{-n D(\hat\mean \| \meanmin)}+ e^{-n D(\hat\mean \| \meanplus)}}  \nonumber \\
& =   \frac{
2 e^{- (n/2) (\hat\theta - \theta)^2}}{e^{-(n/2)\cdot (\hat\mean - \mean + \sqrt{2C/n})^2}+ e^{-(n/2) (\hat\mean - \mean - \sqrt{2C/n})^2}}  
 = \nonumber \\
 & =
\frac{2 e^{C\cdot (n/n^*)}
}{e^{-\sqrt{n} \cdot (\hat\mean - \mean)\sqrt{2 C(n/n^*)}  }+ e^{\sqrt{n} \cdot  (\hat\mean - \mean) \sqrt{2  C (n/n^*)} }}  
\end{align}
We will assume that, for any observed $y$, the MLE $\hat\theta = \hat\theta(y)$ is chosen as action (we will derive that this is the e-posterior minimax action later on, in Proposition~\ref{prop:KLahoy}). Thus, to calculate the risk assessment bound (\ref{eq:riskassessment}) we need to find the maximum over $\theta$ of 
\begin{equation}
    \label{eq:maxima}
(\hat\theta-\theta)^2 \bar{P}(\theta \mid y)
\end{equation} 
We find, using (\ref{eq:gaussrobustness}), and the fact that 
the maximum (\ref{eq:maxima}) must be  achieved at the same $\theta$ as the maximum of $2 C n (n/n^*) (\hat\theta-\theta)^2 \bar{P}(\theta \mid y)$, that the maximizing $\theta_n$ satisfies 
\begin{equation}
    \label{eq:cosh}
\sqrt{2 C n \cdot \frac{n}{n^*} (\theta_n- \hat\theta)^2} = x^* \text{\  where \ } x^*
\text{\ maximizes\ }
\frac{x^2}{\exp(x) + \exp(-x)},
\end{equation}
a function with a unique maximum, found numerically to be achieved at $x^*= \pm 2.065338138969\ldots$. (\ref{eq:maxima}) then becomes,
\begin{equation}
    \label{eq:minima}
\frac{1}{n} \cdot \frac{(x^*)^2}{2 C (n/n^*)} \cdot \frac{2 e^{C (n/n^*)}}{e^{-x^*} + e^{x^*}}
\end{equation}
We want to find the $C$ giving the best possible bounds for the choice $n=n^*$, i.e.\ when the anticipated sample size turns out correct.  
This can be found by minimizing (\ref{eq:minima}) for $n=n^*$: by differentiation we find $C=1$, corresponding to the choice in Example~\ref{ex:normalcontinued}. 
We then find (still at $n=n^*$) that $\theta_n = \hat\theta \pm x^*/\sqrt{ 2n} \approx \hat\theta \pm 1.46/\sqrt{n}$.
Similarly, $\bar{P}(\theta \mid y)$ for the maximizing $\theta_n$ at $n=n^*$ is given by $2 e^{1} /(e^{x^*} + e^{-{x^*}}) = 0.6783206434\ldots$. 
The desired maximum (\ref{eq:maxima}), at general $n$,  is then given, using (\ref{eq:minima}) with $C=1$, by $(1/n)  (n^*/n) e^{n/n^*} x^{*2}/(e^{x^*} + e^{-{x^*}}) =  (1/n)  (n^*/n) e^{n/n^*} \cdot 0.53222078$, which, if $n= n^*$, evaluates to  
$(1/n) \cdot 1.446729604 \ldots $ and 
(\ref{eq:superbound}) follows.

\paragraph{Proof of Theorem~\ref{thm:exp} above and generalized Claim in Example~\ref{ex:normalzero}}
For simplicity we only consider the case with  $n=n^*$ here; the extension to $c \leq n/n^* \leq 1/c$ can then be derived analogously to the Gaussian case. We now let $\Theta$ represent the mean-value parameterization.
As in the main text, we set $\theta^+$ and $\theta^-$ as in (\ref{eq:postval}) above with $C= 1$, which defines $\bar{P}(\theta \mid y)$.
We fix an arbitrary  compact subinterval $\Theta^{\circ}$ of $\Theta$.
We need to bound, for arbitrary fixed $\hat\theta \in \Theta^{\circ}$ the maximum over $\theta \in \Theta^{\circ}$ of
\begin{equation}
    \label{eq:exprisk}
    2 \cdot D(\hat\theta \mid \theta) \bar{P}(\theta \mid y).
\end{equation}
We let $\beta(\theta)$ denote the canonical parameter corresponding to $\theta$. Note that, for $g(\theta^*) = D(\theta^* \| \theta) - D(\theta^* \| \theta^{-})$, we have 
$g'(\theta^*) = \beta(\theta^-) - \beta(\theta) < 0$ constant in $\theta^*$ (here we used that $\beta$ is monotonically increasing in $\theta$). In particular, if $\hat\theta$ and $\theta$ are both contained in $\Theta^{\circ}$, then $g'(\hat\theta) \leq - C^-$ where $C^- = \min _{\theta \in\Theta^{\circ}} \beta(\theta) - \beta(\theta^-) > 0$, and then, 
for $\hat\theta = \theta - a$, $a > 0$:
\begin{align*}
    \bar{P}(\theta \mid y) & \leq 2
 e^{-n (D(\hat\theta \| \theta) - D(\hat\theta \| (\theta)^{-}))}  =  2 e^{-n (D(\theta -a \|\theta) - D(\theta -a \| \theta^-) } \\
 & \leq 2 e^{n D(\theta \|\theta^-) - a n C^* } = 2 e^{C - a n C^-}.
\end{align*}
Analogously, if $\hat\theta,\theta \in \Theta^{\circ}$, $\hat\theta = \theta +a$, $a > 0$, we can derive
$
 \bar{P}(\theta \mid y)  \leq 2  e^{C - a n C^+}
$
for $C^+ = \min_{\theta \in\Theta^{\circ}} \beta(\theta^+) - \beta(\theta) > 0$. It follows that with $C^{\pm} = \min \{C^-, C^+\}$, there exists another constant $C'$ such that, if $\theta,\hat\theta \in \Theta^{\circ}$,
then the risk assessment at $\theta$ satisfies
\begin{equation}\label{eq:ketel}
2\cdot D(\hat\theta \| \theta) \bar{P}(\theta \mid y) \leq b C' a^2 e^{C- |a| n C^{\pm}}.\end{equation}
Writing $|a| = \sqrt{K/n}$, we see that for large enough $K$, the right-hand side is smaller than $1/n$. We thus find that there is a constant $K$ such that, for all $\theta,\hat\theta \in \Theta^\circ$, if $|\hat\theta - \theta|  \geq \sqrt{K/n}$, then the risk bound (\ref{eq:ketel}), and hence (\ref{eq:exprisk}), is smaller than $1/n$.
Without loss of generality we may assume $K \geq 1.46^2$ (the need for this condition will become clear later).

We now  fix $K'> 0$ and set, for $\alpha \in [-K',K']$, for $\hat\theta \in \Theta^{\circ}$,
\begin{equation}
    \label{eq:sproet}
\hat\theta_{\alpha} := \hat\theta + \alpha \cdot \sqrt{\frac{2}{n I(\hat\theta)}}. 
\end{equation}
where $K'$ is chosen such that $\hat\theta_{-K'}\leq \hat\theta - \sqrt{K/n}$ and $\hat\theta_{K'} \geq \hat\theta + \sqrt{K/n}$ for all $\hat\theta \in \Theta^{\circ}$.
Since $I(\theta) $ is bounded from below on $\Theta^{\circ}$ by a positive constant, such a $K'$ must exist (note we do not require that  $\theta_{\alpha} \in \Theta^{\circ}$ for all $\alpha \in [-K', K]$). 
We now define $\Theta'_{\hat\theta} :=  \Theta^{\circ} \cap \{\hat\theta_{\alpha}: \alpha \in [-K',K]\}$.
Our proof strategy for the remainder of the proof is showing that $\Theta'_{\hat\theta}$ is an interval around $\hat\theta$ that is small enough so that we can  re-derive the results for the Gaussian case (up to a factor converging to $1$) on this set, and then showing that the global maximum of (\ref{eq:exprisk}) over $\Theta^{\circ}$ is achieved on this set, so that only this set matters. 

We have that $|I'(\theta)/I(\theta)|$, where $I'$ is the derivative if $I$, is bounded by a constant on $\Theta^{\circ}$. A first-order Taylor approximation of $I(\theta)$ then 
 implies that there exist a constant $C^{\circ}$ such that for all large enough $n$ larger than some $n^{\circ}$, for all $\hat\theta  \in \Theta^{\circ}$, we have
\begin{equation}
    \label{eq:spekkie}
\text{for all $\theta,\theta^{\circ} \in \Theta'_{\hat\theta}$}\ : 
1 -  \sqrt{\frac{C^{\circ}}{n}} \leq  \frac{I(\theta)}{I(\theta^{\circ})} \leq 1 + \sqrt{\frac{C^{\circ}}{n}}.
\end{equation}
We consider, at sample size $n >n^{\circ}$ the risk assessment (\ref{eq:riskassessment}), i.e.\ the maximum over $\theta \in \Theta^{\circ}$ of  (\ref{eq:exprisk}) that we get for any fixed $\hat\theta(y) \in \Theta^{\circ}$. We first consider the maximum over $\theta \in \Theta'_{\hat\theta}$.
%

Within $\Theta'_{\hat\theta}$, we can  perform essentially the same derivation as in (\ref{eq:gaussrobustness}), using the Taylor approximation (\ref{eq:taylor}) in the second step in both numerator and denominator and aggressively and repeatedly using (\ref{eq:spekkie}) above to vary $\theta$ in $I(\theta)$. 
This (tediously but straightforwardly) gives, for $\alpha \in [-K',K']$ such that $\hat\theta_{\alpha} \in\Theta'$,  
\begin{align}
& \bar{P}(\hat\theta_{\alpha} \mid y)  = 2 e^{1}  e^{O(n^{-1/2})} \cdot \frac{1}{e^{-2 \alpha } + e^{ 2\alpha }}  
 = 2 e^{1} e^{O(n^{-1/2})} \cdot 
\frac{1}{e^{- \sqrt{ (2  I(\hat\theta) n }(\theta- \hat\theta) }  + e^{+ \sqrt{( 2  I(\hat\theta) n }(\theta - \hat\theta)} } \nonumber \\ 
& D(\hat\theta \| \hat\theta_{\alpha} ) =  (1 + O(n^{-1/2})) (\hat\theta - \hat\theta_{\alpha})^2 I(\hat\theta) \nonumber
\end{align}
which holds uniformly for all $\hat\theta \in \Theta^{\circ}$ and all $\alpha \in [-K',K']$ such that  $\hat\theta_{\alpha} \in {\Theta}'_{\hat\theta}$, i.e.\ we can pick the same  constant in  $O(n^{-1/2})$ for all such $\hat\theta$ and $\alpha$. 
But this means that for all large enough $n$, on the set $\Theta'$,   we can repeat the same reasoning as in (\ref{eq:cosh})  in the Gaussian case up to a factor of order $1+ O(n^{-1/2})$: since this factor converges to $0$, we must have that, uniformly for all $\hat\theta \in \Theta$, for all $n$ larger than some $n_0$, the maximum will be above $(1/n) \cdot  1.44$ and will be achieved by one or both elements inside a set $\Theta_{\max} = \{ \hat\theta + \gamma^+_n/\sqrt{n}, \hat\theta - \gamma^-_n/\sqrt{n} \} \cap \Theta'$, where both $\gamma^+_n$ and $\gamma^-_n$ converge to $1.46 \ldots$, at which point $\bar{P}(\theta_n \mid y) \approx 0.68$. 
Note that by requiring $K \geq 1.46^2$ we made sure  $K'$ is large enough so that both choices for the maximum are in $\{\hat\theta_{\alpha}: \alpha \in [-K',K]\}$; at least one of these two choices for $\theta_n$ then also lies in $\Theta'$ for $n$ larger than some $n_0$, where $n_0$ can be chosen uniformly, independently of $\hat\theta$. 

Thus, the risk (\ref{eq:exprisk}) maximized over $\theta \in \Theta'_{\hat\theta} \subset \Theta^{\circ}$ is bounded from below by $1.44 /n$, and the maximum over $\Theta^{\circ}$ must be at least as large. But (\ref{eq:ketel}) now gives that the maximum of (\ref{eq:exprisk}) over $\theta \in \Theta^{\circ} \setminus \Theta'_{\hat\theta}$ is at most $1/n$ which is smaller than the maximum risk within $\Theta'_{\hat\theta}$. Therefore the global maximum is achieved within $\Theta'_{\hat\theta}$ and, just as for the Gaussian location case, is given by $(1/n) \cdot 1.45\ldots$,
which is what we had to prove (the extension to $n^* \neq n$ mentioned in (\ref{eq:superboundb}) is now trivial to add; the various `for all large enough $n$' qualifications we imposed are absorbed into $C^*$)). As a by-product, we also find that, since $\bar{P}(\theta_n \mid y) \ll 1$, Condition Zero holds automatically in this setting. 
\commentout{
\paragraph{Proof for Example~\ref{ex:normaldiscrete}}
From (\ref{eq:robustness}) applied with a prior that puts mass $1$ on $\meanmin$ as defined by (\ref{eq:al}), $W(\{\meanmin\}) = 1$, we have, for general regular exponential families, 
$\bar{P}_{[n^*,\alpha],R}(\mean \mid Y) \leq \alpha$ iff  
\begin{equation}\label{eq:KLreject}
 D(\hat\mean \| \mean) -  D(\hat\mean \| \meanmin) \geq \frac{- \log \alph}{n}
\end{equation}
In the normal location example, $D(\hat\mean \|\mean) = (1/2) (\hat\mean - \mean)^2$
and we can rewrite this as 
$$
- (\mean - \meanmin) \left(\hat\mean - \frac{1}{2}(\mean + \meanmin )\right) \geq \frac{- \log \alph}{n}.
$$
Let $c_{\alpha}$ be as in the main text and let us define
$V:= (- \log \alph)/n$. 
In the normal location family, we have by definition of $\meanmin$ that $(\meanmin- \mean)^2/2 = c_{\alpha} V$. Therefore we have
$\mean - \meanmin = \sqrt{ 2c_{\alpha} V}$ and also $(\mean + \meanmin)/2 = \mean - \sqrt{c_{\alpha} V/2}$ giving, after some manipulation, (\ref{eq:normalreject}).}
\section{Establishing that the MLE is e-posterior minimax}\label{app:mleminimax}
The proposition below shows, by setting $\theta' := \hat\theta$, that the MLE  $\hat\theta$ is e-posterior minimax both for the Gaussian location family with the two-point prior (Example~\ref{ex:normalcontinued}) and for general exponential families with a pure Savage-Dickey e-variable with arbitrary prior (Example~\ref{ex:logn}). 
\begin{proposition}\label{prop:KLahoy}
Consider a regular exponential family given in its mean-value parameter space as above. Let $\theta' \in \Theta$ and let $f: \reals^+_0 \rightarrow \reals^+_0$ be a continuous function such that $\lim_{x \rightarrow \infty}   x f(x) = 0$.
We have
$$
\min_a \max_{\mean \in \Theta} D(a \| \mean) \cdot f(D(\mean'\| \mean))
$$
is achieved by $a = \mean'$. In particular, 
\begin{enumerate}
    \item 
Suppose that $\hat\theta = \hat\theta(y) \in \Theta$ and consider the dampened  e-posterior $\bar{P}^{[\gamma]}(\theta \mid y)$  for any $0 <  \gamma\leq 1$ and relative to any prior $W$. It  can be written 
as $f(D(\hat\theta \| \mean))$ for a function $f$ of the required type.
\item For the normal location family, the e-posterior based on the two-point discrete prior as in (\ref{eq:gaussrobustness}) can also be written as $f(D(\hat\theta \| \mean))$ for a function $f$ of the required type.
\end{enumerate}
\end{proposition}
\begin{proof}
Let $g(a) := \max_{\mean \in \Theta} D(a \| \mean) \cdot f(D(\mean'\| \mean))$. 
First consider $a=\theta'$. It follows from (\ref{eq:expfamproperty}) that  the maximum over $\mean$ in the definition of  $g(a) = g(\theta')$ is achieved by some $\mean^*_- < \theta'$ and also by some $\mean^*_+ > \theta'$.  
Now consider $a \neq \mean'$. We must show that $g(a) > g(\theta')$. 
If $a > \mean'$, we have that $D(a\| \mean^*_-)  >   D(\mean'\| \mean^*_-)$ so 
$$g(a) \geq D(a \| \mean^*_-) \cdot f(D(\mean'\| \mean^*_-))
> D(\mean'\| \mean^*_-) \cdot f(D(\mean'\| \mean^*_-)) = g(\theta'). 
$$ 
The case $ a< \mean'$ goes similarly, with $\mean^*_+$ replacing $\mean^*_-$.
This establishes the first result.

As to (1), the case with $\gamma=1 $ now follows directly from (\ref{eq:robustness}). For $\gamma < 1$, use the fact that $1/((1-\gamma) x^{-1} + \gamma)$ is increasing in $x$.

As to (2):
using (\ref{eq:gaussrobustness}), 
and once more that $D(\hat\theta \| \theta) = (1/2) (\hat\theta - \theta)^2$, and considering separately the cases that $\hat\theta > \theta$ and $\hat\theta < \theta$,  we find that
$
\bar{P}(\theta \mid y)   = f(D(\hat\theta \|\theta)),
$
where $f$ is of the required form. 
\end{proof}

\section{Pure Savage-Dickey based E-Collections for Gaussian and Exponential Families} 
\label{app:logn}
Let $\bar{P}^{[\gamma]}(\theta \mid y)$ be the `dampened e-posterior' as in the main text, based on arbitrary prior $W$ on $\Theta$.  
In Proposition~\ref{prop:KLahoy} above we showed that for any $0 <  \gamma \leq 1$ the MLE $\hat\theta$ is the $\bar{P}^{[\gamma]}$-e-posterior minimax estimator irrespective of $b$. Below we further, via Proposition~\ref{prop:dampening} below, show that, for the choice $\gamma = 1/2$, the risk assessment bound (\ref{eq:riskassessment}) holds with 
\begin{equation}\label{eq:boundyboundy}
\bar{R}(y,\hat\theta(y)) = \frac{2  }{n} D(P_{\hat\theta}^{(n)} \| P_W^{(n)}) 
\end{equation}
for all $n$ such that the expression on the right is larger than $1$ --- which will be the case for all but the smallest $n$. 
Here  we used the notation $D(P_{\theta}^{(n)} \| P_{W}^{(n)})$ for  the KL divergence between distribution $P_{\theta}$ and Bayes marginal $P_W$, both defined on $n$ outcomes. For the normal location family  with prior with mean $0$ and precision $\lambda$, we have the exact expression 
\begin{equation}
    \label{eq:houndy}
D(P_{\hat\theta}^{(n)} \| P_W^{(n)}) =  
\frac{1}{2} \log \frac{n + \lambda}{\lambda} + \frac{n \lambda}{n+\lambda} \hat\theta^2 
\end{equation}
which together with (\ref{eq:boundyboundy}) gives (\ref{eq:boundy}) in the main text. (\ref{eq:houndy}) is found by using the fact that $D(P_{\hat\theta}^{(n)} \| P_W^{(n)})= \ln \bar{P}_{W}(\hat\theta \mid y)$, an identity which follows from (\ref{eq:robustness})  and holds for general regular exponential families. If these have  continuous prior $w$, we get, for $\hat\theta$ in any compact subset of the parameter space, the expression
(Chapter 8 of \cite{Grunwald07})
$$
D(P_{\hat\theta}^{(n)} \| P_W^{(n)}) =  
\frac{1}{2} \log \frac{n}{2 \pi}  - \log \frac{w(\hat\theta)}{I(\hat\theta)^{1/2}} + o(1),
$$
with $I(\hat\theta)$ the Fisher information at $\hat\theta$, allowing us to generalize (\ref{eq:boundy}) to general exponential families.

We proceed to derive (\ref{eq:boundyboundy}). We start with a proposition that is applicable more generally than just for KL  loss: 
\begin{proposition}\label{prop:dampening}
Let  $\Theta \subset \reals$, $b \in \reals^+$ and let $L_{\currz}: \Theta \times \Theta \rightarrow \reals^+_0$ be a loss function  and let  $\breve\theta: \cY \rightarrow \Theta$ be an estimator.
Fix some $y \in \cY$ and let $\breve\theta := \breve\theta(y)$.  
Consider an e-posterior $\bar{P}(\theta \mid y)$ with $\bar{P}_{\sup} := \sup_{y \in \cY,\theta \in \Theta} \bar{P}(\theta \mid y)$.
and let $\bar{P}'$ be an upper bound on $\bar{P}$ up to a factor $C_{\sup}$ i.e.\ for all $\theta \in \Theta$, $y \in \cY$, $ \bar{P}(\theta \mid y) \leq C_{\sup} \bar{P}'(\theta \mid y)$. 
Fix some $\theta_L, \theta_R \in \Theta$ (depending on $\breve\theta$) with $\theta_L < \theta_R$ so that both: 
\begin{enumerate}
    \item 
for all $\theta \geq \theta_R$, $\bar{P}'(\theta \mid y) \leq 1$ and $\bar{P}'(\theta \mid y) L(\theta,\breve\theta)$ is decreasing in $\theta$.
 \item 
for all $\theta \leq \theta_L$, $\bar{P}'(\theta \mid y) \leq 1$ and $\bar{P}'(\theta \mid y) L(\theta,\breve\theta)$ is increasing in $\theta$. 
\end{enumerate} 
Then  our main risk assessment bound (\ref{eq:riskassessment}) holds for  $\bar{P}(\theta \mid y)$
with
\begin{equation}
    \label{eq:twice}
\bar{R}(y,\breve\theta(y)) =   b(y) \cdot \max\{ C_{\sup}, \bar{P}_{\sup} \} \cdot \max_{\theta \in [\theta_L,\theta_R]}  L (\theta,\breve\theta).
\end{equation}
\end{proposition}
\begin{proof}
For $\theta \leq \theta_L$, $\bar{P}(\theta \mid y) L_{\currz}(\theta,\breve\theta) \leq 
\bar{P}'(\theta \mid y) C_{\sup} L_{\currz}(\theta,\breve\theta) \leq C_{\sup} \bar{P}'(\theta_L \mid y) L_{\currz}(\theta_L,\breve\theta) \leq C_{\sup} L_{\currz}(\theta_L,\breve\theta)$.
Analogously for $\theta \geq \theta_R$, we have  $\bar{P}(\theta \mid y) L_{\currz}(\theta,\breve\theta) \leq  C_{\sup} L_{\currz}(\theta_R,\breve\theta)$.
Finally for $\theta \in [\theta_L,\theta_R]$, we have 
$\bar{P}(\theta \mid y) L_{\currz}(\theta,\breve\theta) \leq \bar{P}_{\sup} \cdot  \max_{\theta \in [\theta_L,\theta_R]} \cdot  L_{\currz} (\theta,\breve\theta)$.
The result follows.
\end{proof}
We now use Proposition~\ref{prop:dampening} to show the bound (\ref{eq:boundyboundy}). Assume the setting of that bound. 
We set $\breve\theta$  to the MLE and $L(\theta,\breve\theta) := D(\breve\theta \| \theta)$. We first apply Proposition~\ref{prop:dampening} with $\bar{P}'(\theta |y) := \bar{P}(\theta |y)$ (and $C_{\sup} =1 )$ set to the dampened e-posterior $\bar{P}^{[1/2]}$ relative to a prior $W$  (independent of $\theta$).  Because of the dampening with $\gamma = 1/2$, we know that $\bar{P}_{\sup} \leq 2$. 
We will apply the proposition with $\theta_L < \theta_R$ such that $\bar{P}^{[1/2]}(\theta_L \mid y) = \bar{P}^{[1/2]}(\theta_R \mid y)=1$. These must exist (use (\ref{eq:expfamproperty})), and by  (\ref{eq:robustness}) they satisfy
$$
D(\hat\theta \| \theta_L) = D(\hat\theta \| \theta_R) = \frac{D(P^{(n)}_{\hat\theta} \| P^{(n)}_W)}{n}.
$$
To verify the conditions of Proposition~\ref{prop:dampening}, we will show, using (\ref{eq:robustness}), that 
\begin{equation}\label{eq:lossbound}
D(\hat\theta  \| \theta)  \exp(-n D(\hat\theta \| \theta) + 
D(P^{(n)}_{\hat\theta} \| P^{(n)}_W)) 
\end{equation}
is increasing for $\theta < \theta_L$ and decreasing for $\theta > \theta_R$. 
For this, setting $C= D(P^{(n)}_{\hat\theta} \| P^{(n)}_W)$, it is sufficient to show that $g(u) := u \exp(-n u + C)$ is decreasing if $u \geq D(\hat\theta_L \| \theta_L)$, i.e.\ if $u \geq C/n$.
Differentiation gives that  $g(u)$ is decreasing if $u > 1/n$, so Proposition~\ref{prop:dampening} can be applied if $C\geq 1$ and then (\ref{eq:twice}) gives (\ref{eq:boundyboundy}). 

\commentout{We next apply Proposition~\ref{prop:dampening} to the dampened discrete e-posterior $\bar{P} := \bar{P}^{[1/2]}_{[n^*,\alpha^*]}$  with the Gaussian location family to show (\ref{eq:hyperbound}). Again, by the dampening, $\bar{P}_{\sup} \leq 2$. 
We will use Proposition~\ref{prop:dampening} with $\bar{P}'(\theta |y) = 2 \exp(n U^2/2 - n |\hat\theta - \theta| U)$ with $U$ as in (\ref{eq:gaussrobustness}) ; $\bar{P}'$ can be seen to be an upper bound of $\bar{P}_{[n^*,\alpha^*]}$ by (\ref{eq:gaussrobustness}) and hence, 
since trivially $\bar{P}^{[1/2]}_{[n^*,\alpha^*]} \leq 2 \bar{P}_{[n^*,\alpha^*]}$, we have $\bar{P}' \leq C_{\sup} \bar{P}^{[1/2]}_{[n^*,\alpha^*]}$ with $C_{\sup} = 2$. 
Using (\ref{eq:prereject}) and (\ref{eq:normalrejectb}) with $\alpha=1$, we find that with $\theta'_L =\hat\theta - V'$, $\theta'_R= \hat\theta+V'$ and $c = (n^*/n) \cdot ((\log 2)/(- \log (\alpha^*/2)))$ and 
$$V'= \sqrt{\frac{\log 2}{2 n}} \cdot \left( c^{1/2}+ c^{-1/2}  \right)$$ 
we are guaranteed that
$\bar{P}'(\theta \mid y) \leq 1$ for $\theta \leq \theta_L$ and $\theta \geq \theta_R$.
Further, simple differentiation shows that  $\bar{P}'(\theta \mid y) (\theta - \hat\theta)^2$ is increasing at $\theta < \theta''_L$ and decreasing at $\theta > \theta''_R$ where $\theta''_L =\hat\theta -V''$ and $\theta''_R =\hat\theta+ V''$ with $$V''= \frac{2}{n U}  = \sqrt{\frac{2}{n}} \cdot \sqrt{\frac{n^*}{n (- \log (\alpha^*/2))}}= \sqrt{\frac{2}{(\log 2) \cdot n}} \cdot c^{1/2}.
$$
Combining these two displays, we find that the conditions of Proposition~\ref{prop:dampening} hold if we set $\theta_L = \hat\theta -V$, $\theta_R = \hat\theta + V$, $V = \sqrt{\frac{2}{(\log 2) \cdot n}} \cdot \left( c^{1/2}+ c^{-1/2}  \right)$. Using $\sup_{\theta \in [\theta_L,\theta_R]} D(\hat\theta \| \theta) = V^2/2$ and $\max\{C_{\sup},p_{\sup} \}= 2$ in (\ref{eq:twice}) now gives 
(\ref{eq:hyperbound}).}



\commentout{

}
\end{document}

%% file: declarations.tex
\usepackage{epsf}
\usepackage{verbatim}
\usepackage{makeidx}
\usepackage{graphicx}
\usepackage{amssymb}
\usepackage{MnSymbol}
\usepackage{amsmath}
\usepackage{latexsym}
\usepackage{amsbsy}
\usepackage{epsfig}

\newtheorem{theorem}{Theorem}

\newtheorem{proposition}{Proposition}

\newtheorem{definition}{Definition}

\newcommand{\commentout}[1]{}

\newcommand{\prs}{\par}

\newenvironment{proof}{\noindent{\bf Proof}:}{\mbox{$\Box$}
\prs\vspace{3mm}\prs}

\newcounter{myenumeratecounter}

\newtheorem{examplehidden}{Example}
\newenvironment{example}{
\begin{examplehidden}
\em
}
{\end{examplehidden}}

\newcommand{\cA}{\ensuremath{\mathcal A}}

\newcommand{\cH}{\ensuremath{\mathcal H}}

\newcommand{\cP}{\ensuremath{\mathcal P}}

\newcommand{\cS}{\ensuremath{\mathcal S}}

\newcommand{\cU}{\ensuremath{\mathcal U}}

\newcommand{\cY}{\ensuremath{\mathcal Y}}

\newcommand{\reals}{{\mathbb R}}

\newcommand{\Exp}{\ensuremath{\text{\rm E}}}



















%% file: mainformatted.bbl
\begin{thebibliography}{54}
\providecommand{\natexlab}[1]{#1}
\providecommand{\url}[1]{\texttt{#1}}
\expandafter\ifx\csname urlstyle\endcsname\relax
  \providecommand{\doi}[1]{doi: #1}\else
  \providecommand{\doi}{doi: \begingroup \urlstyle{rm}\Url}\fi

\bibitem[Balch et~al.(2019)Balch, Martin, and Ferson]{balch2019satellite}
Michael~Scott Balch, Ryan Martin, and Scott Ferson.
\newblock Satellite conjunction analysis and the false confidence theorem.
\newblock \emph{Proceedings of the Royal Society A}, 475\penalty0
  (2227):\penalty0 20180565, 2019.

\bibitem[Barndorff-Nielsen(1978)]{BarndorffNielsen78}
O.E. Barndorff-Nielsen.
\newblock \emph{Information and Exponential Families in Statistical Theory}.
\newblock Wiley, Chichester, UK, 1978.

\bibitem[Bates et~al.(2022)Bates, Jordan, Sklar, and Soloff]{BatesJSS22}
S.~Bates, M.~I. Jordan, M.~Sklar, and J.~Soloff.
\newblock Principal-agent hypothesis testing.
\newblock \emph{arXiv:2205.06812}, 2022.

\bibitem[Berger(2003)]{Berger03}
J.~Berger.
\newblock Could {F}isher, {J}effreys and {N}eyman have agreed on testing?
\newblock \emph{Statistical Science}, 18\penalty0 (1):\penalty0 1--12, 2003.

\bibitem[Berger et~al.(2022)Berger, Bernardo, and Sun]{BergerBS22}
J.~Berger, J.~Bernardo, and D.~Sun.
\newblock Objective {B}ayesian inference and its relationship to frequentism.
\newblock In \emph{Handbook of Bayesian, Fiducial, and Frequentist Inference}.
  Blackwell, 2022.

\bibitem[Berger and Guglielmi(2001)]{BergerG01}
James~O. Berger and Alessandra Guglielmi.
\newblock Bayesian and conditional frequentist testing of a parametric model
  versus nonparametric alternatives.
\newblock \emph{Journal of the American Statistical Association}, 96\penalty0
  (453):\penalty0 174--184, 2001.

\bibitem[Berger and Wolpert(1988)]{BergerW88}
J.O. Berger and R.L. Wolpert.
\newblock \emph{The Likelihood Principle}.
\newblock Institute of Mathematical Statistics, Hayward, CA, 2nd edition, 1988.

\bibitem[Berger et~al.(1994)Berger, Brown, and Wolpert]{BergerBW94}
J.O. Berger, L.D. Brown, and R.L. Wolpert.
\newblock A unified conditional frequentist and {B}ayesian test for fixed and
  sequential simple hypothesis testing.
\newblock \emph{Annals of Statistics}, 22\penalty0 (4):\penalty0 1787--1807,
  1994.

\bibitem[Bernardo and Smith(1994)]{BernardoS94}
J.M. Bernardo and A.F.M Smith.
\newblock \emph{Bayesian Theory}.
\newblock Wiley, Chichester, 1994.

\bibitem[Brownie and Kiefer(1977)]{brownie1977ideas}
C~Brownie and J~Kiefer.
\newblock The ideas of conditional confidence in the simplest setting.
\newblock \emph{Communications in Statistics-Theory and Methods}, 6\penalty0
  (8):\penalty0 691--751, 1977.

\bibitem[Darling and Robbins(1967)]{darling1967confidence}
D.A. Darling and H.~Robbins.
\newblock Confidence sequences for mean, variance, and median.
\newblock \emph{Proceedings National Academy of Sciences}, 58\penalty0
  (1):\penalty0 66, 1967.

\bibitem[Dass and Berger(2003)]{DassB03}
Sarrat Dass and James Berger.
\newblock Unified conditional frequentist and bayesian testing of composite
  hypotheses.
\newblock \emph{Scandinavian Journal of Statistics}, 2003.

\bibitem[Edwards et~al.(1963)Edwards, Lindman, and Savage]{EdwardsLS63}
W.~Edwards, H.~Lindman, and L.J. Savage.
\newblock Bayesian statistical inference for psychological research.
\newblock \emph{Psychological Review}, 70:\penalty0 193--242, 1963.

\bibitem[Good(1991)]{Good91}
I.~J. Good.
\newblock C383. a comment concerning optional stopping.
\newblock \emph{Journal of Statistical Computation and Simulation}, 39\penalty0
  (3):\penalty0 191--–192, 1991.

\bibitem[Gr\"unwald(2007)]{Grunwald07}
P.~Gr\"unwald.
\newblock \emph{The Minimum Description Length Principle}.
\newblock MIT Press, Cambridge, MA, 2007.

\bibitem[Gr{\"u}nwald and Mehta(2019)]{GrunwaldM19}
P.~Gr{\"u}nwald and N.~Mehta.
\newblock A tight excess risk bound via a unified
  {P}{A}{C}-{B}ayesian-{R}ademacher-{S}htarkov-{M}{D}{L} complexity.
\newblock In \emph{Proceedings of the Thirtieth Conference on Algorithmic
  Learning Theory (ALT) 2019}, 2019.

\bibitem[Gr\"unwald and van Ommen(2017)]{GrunwaldO17}
P.~Gr\"unwald and T.~van Ommen.
\newblock Inconsistency of {B}ayesian inference for misspecified linear models,
  and a proposal for repairing it.
\newblock \emph{Bayesian Analysis}, 12\penalty0 (4):\penalty0 1069--1103, 2017.

\bibitem[Gr{\"u}nwald et~al.(2019)Gr{\"u}nwald, de~Heide, and
  Koolen]{GrunwaldHK19}
P.~Gr{\"u}nwald, Rianne de~Heide, and Wouter Koolen.
\newblock Safe testing, 2019.
\newblock arXiv preprint arXiv:1906.07801 . Accepted pending minor
  modifications to Journal of the Royal Statistical Society, Series B.

\bibitem[Gr\"unwald(1999)]{Grunwald99a}
P.~D. Gr\"unwald.
\newblock Viewing all models as ``probabilistic''.
\newblock In \emph{Proceedings of the Twelfth ACM Conference on Computational
  Learning Theory (COLT' 99)}, pages 171--182, 1999.

\bibitem[Gr\"unwald and Halpern(2011)]{GrunwaldH11}
P.D. Gr\"unwald and J.Y. Halpern.
\newblock Making decisions using sets of probabilities: Updating, time
  consistency, and calibration.
\newblock \emph{Journal of Artificial Intelligence Research (JAIR)},
  42:\penalty0 393--426, 2011.

\bibitem[Gr\"unwald(2018)]{grunwald2018safe}
Peter Gr\"unwald.
\newblock Safe probability.
\newblock \emph{Journal of Statistical Planning and Inference}, 2018.

\bibitem[Gr{\"u}nwald(2022)]{Grunwald22}
Peter Gr{\"u}nwald.
\newblock Beyond {N}eyman-{P}earson.
\newblock \emph{arXiv:2205.00901}, 2022.

\bibitem[Hendriksen et~al.(2021)Hendriksen, de~Heide, and
  Gr{\"u}nwald]{hendriksen2021optional}
Allard Hendriksen, Rianne de~Heide, and Peter Gr{\"u}nwald.
\newblock Optional stopping with bayes factors: a categorization and extension
  of folklore results, with an application to invariant situations.
\newblock \emph{Bayesian Analysis}, 16\penalty0 (3):\penalty0 961--989, 2021.

\bibitem[Henzi and Ziegel(2021)]{henzi2021valid}
Alexander Henzi and Johanna~F. Ziegel.
\newblock Valid sequential inference on probability forecast performance.
\newblock \emph{arXiv preprint arXiv:2103.08402}, 2021.

\bibitem[Herbrich and Williamson(2002)]{HerbrichW02a}
Ralf Herbrich and Robert~C. Williamson.
\newblock Algorithmic luckiness.
\newblock In \emph{Advances in Neural Information Processing Systems},
  volume~14, 2002.

\bibitem[Howard et~al.(2020)Howard, Ramdas, McAuliffe, and
  Sekhon]{howard2020time}
Steven~R Howard, Aaditya Ramdas, Jon McAuliffe, and Jasjeet Sekhon.
\newblock Time-uniform {C}hernoff bounds via nonnegative supermartingales.
\newblock \emph{Probability Surveys}, 17:\penalty0 257--317, 2020.

\bibitem[Howard et~al.(2021)Howard, Ramdas, McAuliffe, and
  Sekhon]{howard2021time}
Steven~R Howard, Aaditya Ramdas, Jon McAuliffe, and Jasjeet Sekhon.
\newblock Time-uniform, nonparametric, nonasymptotic confidence sequences.
\newblock \emph{The Annals of Statistics}, 49\penalty0 (2):\penalty0
  1055--1080, 2021.

\bibitem[Johnson(2013)]{johnson2013uniformly}
Valen~E Johnson.
\newblock Uniformly most powerful {B}ayesian tests.
\newblock \emph{Annals of statistics}, 41\penalty0 (4):\penalty0 1716, 2013.

\bibitem[Kiefer(1977)]{Kiefer77}
J.~Kiefer.
\newblock Conditional confidence statements and confidence estimators.
\newblock \emph{Journal of the American Statistical Association}, 72\penalty0
  (360):\penalty0 789--808, 1977.

\bibitem[Kiefer(1976)]{kiefer1976admissibility}
Jack Kiefer.
\newblock Admissibility of conditional confidence procedures.
\newblock \emph{The Annals of Statistics}, pages 836--865, 1976.

\bibitem[Neiswanger and Ramdas(2021)]{neiswanger2021uncertainty}
Willie Neiswanger and Aaditya Ramdas.
\newblock Uncertainty quantification using martingales for misspecified
  {G}aussian processes.
\newblock In \emph{Algorithmic Learning Theory}, pages 963--982. PMLR, 2021.

\bibitem[Neyman(1950)]{Neyman50}
J.~Neyman.
\newblock \emph{First Course in Probability and Statstics}.
\newblock Henry Holt and Company, New York, 1950.

\bibitem[Oelrich et~al.(2020)Oelrich, Ding, Magnusson, Vehtari, and
  Villani]{oelrich2020overconfident}
Oscar Oelrich, Shutong Ding, Måns Magnusson, Aki Vehtari, and Mattias Villani.
\newblock When are {B}ayesian model probabilities overconfident?
\newblock \emph{arXiv preprint arXiv:2003.04026}, 2020.

\bibitem[Pawel et~al.(2022)Pawel, Ly, and Wagenmakers]{PawelLW22}
Samuel Pawel, Alexander Ly, and Eric-Jan Wagenmakers.
\newblock Evidential calibration of confidence intervals.
\newblock \emph{arXiv:2206.12290}, 2022.

\bibitem[P{\'e}rez-Ortiz et~al.(2022)P{\'e}rez-Ortiz, Lardy, de~Heide, and
  Gr{\"u}nwald]{perez2022estatistics}
Muriel~Felipe P{\'e}rez-Ortiz, Tyron Lardy, Rianne de~Heide, and Peter
  Gr{\"u}nwald.
\newblock E-statistics, group invariance and anytime valid testing.
\newblock \emph{arXiv:2208.07610}, 2022.

\bibitem[Ramdas et~al.(2022)Ramdas, Ruf, Larsson, and
  Koolen]{ramdas2022testing}
Aaditya Ramdas, Johannes Ruf, Martin Larsson, and Wouter~M Koolen.
\newblock Testing exchangeability: Fork-convexity, supermartingales and
  e-processes.
\newblock \emph{International Journal of Approximate Reasoning}, 141:\penalty0
  83--109, 2022.

\bibitem[Ren and Barber(2022)]{RenB22}
Zhimei Ren and Rina~Foygel Barber.
\newblock Derandomized knockoffs: Leveraging e-values for false discovery rate
  control.
\newblock \emph{arXiv:2205.15461}, 2022.

\bibitem[Robins and Wasserman(2012)]{Larry2012}
J.~Robins and L.~Wasserman.
\newblock Robins and {W}asserman respond to a {N}obel prize winner, 2012.
\newblock URL
  \url{https://normaldeviate.wordpress.com/2012/08/28/robins-and-wasserman-respond-to-a-nobel-prize-winner/}.
\newblock {B}log post on the {\em Normal Deviate\/} blog.

\bibitem[Royall(1997)]{royall1997statistical}
Richard Royall.
\newblock \emph{Statistical evidence: a likelihood paradigm}.
\newblock Chapman and Hall, 1997.

\bibitem[Schweder and Hjort(2016)]{SchwederH16}
T.~Schweder and N.~Hjort.
\newblock \emph{Confidence, Likelihood, Probability: Statistical Inference with
  Confidence Distributions}.
\newblock Cambridge University Press, 2016.

\bibitem[Shafer and Vovk(2019)]{ShaferV19}
G.~Shafer and V.~Vovk.
\newblock \emph{Game-Theoretic Probability: Theory and Applications to
  Prediction, Science and Finance}.
\newblock Wiley, 2019.

\bibitem[Shafer(2021)]{Shafer:2021}
Glenn Shafer.
\newblock Testing by betting: a strategy for statistical and scientific
  communication (with discussion and response).
\newblock \emph{Journal of the Royal Statistic Society A}, 184\penalty0
  (2):\penalty0 407--478, 2021.

\bibitem[Shawe-Taylor and Williamson(1997)]{ShaweTaylorW97}
J.~Shawe-Taylor and R.C. Williamson.
\newblock A {PAC} analysis of a {B}ayesian classifier.
\newblock In \emph{Proceedings of the Tenth ACM Conference on Computational
  Learning Theory (COLT' 98)}, pages 2--9, Nashville, Tennessee, 1997.

\bibitem[Szab{\'o} et~al.(2015)Szab{\'o}, Van Der~Vaart, and van
  Zanten]{szabo2015frequentist}
Botond Szab{\'o}, Aad~W Van Der~Vaart, and JH~van Zanten.
\newblock Frequentist coverage of adaptive nonparametric {B}ayesian credible
  sets.
\newblock \emph{The Annals of Statistics}, 43\penalty0 (4):\penalty0
  1391--1428, 2015.

\bibitem[Ter~Schure et~al.(2021)Ter~Schure, Perez-Ortiz, Ly, and
  Gr{\"u}nwald]{TerschurePLG21}
J.~Ter~Schure, M.F. Perez-Ortiz, A.~Ly, and P.~Gr{\"u}nwald.
\newblock The safe log rank test: Error control under continuous monitoring
  with unlimited horizon.
\newblock \emph{arXiv:1906.07801}, 2021.

\bibitem[Turner and Gr{\"u}nwald(2022)]{TurnerG22}
Rosanne Turner and Peter Gr{\"u}nwald.
\newblock Anytime-valid confidence intervals for contingency tables and beyond.
\newblock \emph{arXiv:2203.09785}, 2022.

\bibitem[Turner et~al.(2021)Turner, Ly, and Gr{\"u}nwald]{TurnerLG21}
Rosanne Turner, Alexander Ly, and Peter Gr{\"u}nwald.
\newblock Generic e-variables for exact sequential k-sample tests that allow
  for optional stopping.
\newblock \emph{arXiv:2106.02693}, 2021.

\bibitem[Vovk(1993)]{Vovk93}
V.G. Vovk.
\newblock A logic of probability, with application to the foundations of
  statistics.
\newblock \emph{Journal of the Royal Statistical Society, series B},
  55:\penalty0 317--351, 1993.
\newblock (with discussion).

\bibitem[Vovk and Wang(2021)]{VovkW21}
Vladimir Vovk and Ruodu Wang.
\newblock E-values: Calibration, combination, and applications.
\newblock \emph{Annals of Statistics}, 2021.

\bibitem[Wald(1939)]{Wald39}
Abraham Wald.
\newblock Contributions to the theory of statistical estimation and testing
  hypotheses.
\newblock \emph{Annals of Mathematical Statistics}, 10:\penalty0 299--326,
  1939.

\bibitem[Wasserman et~al.(2020)Wasserman, Ramdas, and
  Balakrishnan]{wasserman2020universal}
Larry Wasserman, Aaditya Ramdas, and Sivaraman Balakrishnan.
\newblock Universal inference.
\newblock \emph{Proceedings of the National Academy of Sciences}, 117\penalty0
  (29):\penalty0 16880--16890, 2020.

\bibitem[Waudby-Smith and Ramdas(2022)]{waudby2020estimating}
Ian Waudby-Smith and Aaditya Ramdas.
\newblock Estimating means of bounded random variables by betting.
\newblock \emph{Journal of the Royal Statistical Society: Series B (Statistical
  Methodology)}, 2022.
\newblock (accepted, to appear with discussion).

\bibitem[Williams(1991)]{Williams91}
D.~Williams.
\newblock \emph{Probability with Martingales}.
\newblock Cambridge Mathematical Textbooks, 1991.

\bibitem[Wolpert(1996)]{Wolpert96b}
R.L. Wolpert.
\newblock Testing simple hypotheses.
\newblock In H.H. Bock and W.~Polasek, editors, \emph{Data Analysis and
  Information Systems: Statistical and Conceptual Approaches}, pages 289--297.
  Springer, Berlin, 1996.

\end{thebibliography}


\begin{thebibliography}{9}

\bibitem{1} Allwood JM, Cullen JM. 2011 \textit{Sustainable materials:  with both eyes open}.
Cambridge, UK: UIT Cambridge. See \href{http://www.withbotheyesopen.com}{http://www.withbotheyesopen.com}.

\bibitem{2}  MacKay DJC. 2008  \textit{Sustainable energy:  without the hot air}.
 Cambridge, UK: UIT Cambridge. See \href{http://www.withouthotair.com}{http://www.withouthotair.com}.

\bibitem{3} Gallman PG. 2011  \textit{Green alternatives and national energy strategy: the facts
 behind the headlines}.  Baltimore,\ MD: Johns Hopkins University Press.

\bibitem{4} MacKay DJC. 2013.  Solar energy in the context of energy use, energy transportation, and
 energy storage. \textit{Proc. R. Soc. A} \textbf{371}.

\end{thebibliography}
